\documentclass{article}
\usepackage[utf8]{inputenc}

\title{Efficient bound preserving and asymptotic preserving semi-implicit schemes for the fast reaction--diffusion system}
\author{Yu Zhao\thanks{Corresponding author.}~\thanks{School of Mathematics, Statistics and Mechanics, Beijing University of Technology, Beijing, 100124, P.R.China. \\  
	E-mail: y.zhao@bjut.edu.cn}~, 
	Zhennan Zhou\thanks{Institute for Theoretical Sciences, Westlake University, Hangzhou, Zhejiang Province, 310030, P.R.China. \\
		E-mail: zhouzhennan@westlake.edu.cn} }
\date{\today}

\usepackage{latexsym}
\usepackage{amssymb,amsbsy,amsmath,amsfonts,amssymb,amscd}
\usepackage{mathrsfs,mathabx}
\usepackage{epsfig, graphicx}
\usepackage{graphics,epsfig}
\usepackage{xcolor,url}
\usepackage{subfigure}
\usepackage{booktabs}
\usepackage{verbatim}  

\usepackage[numbers,sort&compress]{natbib}

\usepackage{amsthm}
\usepackage{amsfonts}
\usepackage{graphicx}
\usepackage{subfigure}
\usepackage{amsmath}
\usepackage{amssymb}
\usepackage{amsmath}
\usepackage{bm}
\usepackage{cases}

\newtheorem{theorem}{Theorem}[section]

\newtheorem{lemma}{Lemma}[section]

\begin{document}
	\maketitle
\begin{abstract}
	We consider a special type of fast reaction--diffusion systems in which the coefficients of the reaction terms of the two substances are much larger than those of the diffusion terms while the diffusive motion to the substrate is negligible. 
	Specifically speaking, the rate constants of the reaction terms are $O(1/\epsilon)$ while the diffusion coefficients are $O(1)$ where the parameter $\epsilon$ is small.
	When the rate constants of the reaction terms become highly large, i.e. $\epsilon$ tends to 0, the singular limit behavior of such a fast reaction--diffusion system is inscribed by the Stefan problem with latent heat, which brings great challenges in numerical simulations. 
	In this paper, we adopt a semi-implicit scheme, which is first-order accurate in time and can accurately approximate the interface propagation even when the reaction becomes extremely fast, that is to say, the parameter $\epsilon$ is sufficiently small.  
	The scheme satisfies the positivity, bound preserving properties and has $L^2$ stability and the linearized stability results of the system. 
	For better performance on numerical simulations, we then construct a semi-implicit Runge--Kutta scheme which is second-order accurate in time. 
	Numerous numerical tests are carried out to demonstrate the properties, such as the order of accuracy, positivity and bound preserving, the capturing of the sharp interface with various $\epsilon$ and to simulate the dynamics of the substances and the substrate, and to explore the heat transfer process, such as solid melting or liquid solidification in two dimensions. 
\end{abstract}

{\small 
	
}

\section{Introduction}

The reaction--diffusion equations are a class of partial differential equations that describe the dynamics of a system in which some quantity, such as the concentration of a chemical, diffuses through space while undergoing some kind of reaction including a chemical reaction. 
Such equations have a wide range of promising applications to describe phenomena arising from natural and social sciences, such as the population dynamics \cite{kareiva1990population}, the chemical reactions \cite{johnson2011original,tyson2013belousov}, the neuronal responses \cite{hodgkin1952quantitative,izhikevich2006fitzhugh} and the enterprise competition \cite{wangersky1978lotka}.
The fast reaction--diffusion system is a special case of reaction--diffusion systems in which the strength of the reaction term of chemical substances is much greater than that of the diffusion term.
It has been widely applied to describe cancer, wound healing, and the regeneration of tissues and bones, and has aroused great research interest in the fields of medicine and biology  \cite{mcfadden1988mathematical}. 
For instance, as for the penetration of radiolabeling antibodies into tumor tissues, the penetration rate of antibodies is very slow 
due to their large size, while antibodies attach to antigens and react quickly, so the whole process can be described by a fast reaction--diffusion system \cite{hilhorst1996fast}. 

In this work, we shall be concerned with a special type of fast reaction--diffusion systems, in which there are two reactants and one substrate. 
In such systems, the two reactants' dynamics in concentration are determined by both the regular diffusion and fast reactions, while the evolution of the substrate's concentration is governed only by its fast reaction.
The rate constants of the reactions are $O(1/\epsilon)$ and the diffusion coefficients are $O(1)$ where the parameter $\epsilon$ is sufficiently small. 
Specifically, we consider the following semi-linear equations: 
\begin{align}
\partial_t u_\epsilon - d_1 \Delta u_\epsilon &= -\frac{1}{\epsilon} u_\epsilon \left[v_\epsilon + \lambda \left(1 - p_\epsilon\right)\right], ~~x \in \mathbb{R}^d, ~~t>0,  \label{eq:u} \\
\partial_t v_\epsilon - d_2 \Delta v_\epsilon &= -\frac{1}{\epsilon} v_\epsilon \left(u_\epsilon + \lambda p_\epsilon\right), \label{eq:v} \\
\partial_t p_\epsilon &= \frac{1}{\epsilon} \left[\left(1 - p_\epsilon\right) u_\epsilon - v_\epsilon p_\epsilon\right], \label{eq:p}
\end{align}
where $u_\epsilon = u_\epsilon(x, t)$ and $v_\epsilon = v_\epsilon(x, t)$ represent the concentrations of the two reactants and $p_\epsilon = p_\epsilon(x, t)$  represents the concentration of the substrate with $x \in \mathbb{R}^d$ indicating the location and $t > 0$ indicating the time \cite{hilhorst2003vanishing,perthame2015parabolic}.
$\lambda > 0$ is a constant, and $d_1, d_2 >0$ are the diffusion coefficients of the two reactants. 
The initial condition of the equations \eqref{eq:u}-\eqref{eq:p} is  
\begin{align}
u_\epsilon(x, 0) = u_\epsilon^0(x), \quad v_\epsilon(x, 0) = v_\epsilon^0(x), \quad p_\epsilon(x, 0) = p_\epsilon^0(x), \label{i.c.}
\end{align}
where $u_\epsilon^0(x), v_\epsilon^0(x)$ and $p_\epsilon^0(x)$ are given.

Following \cite{perthame2015parabolic}, there are several main properties for the fast reaction--diffusion system \eqref{eq:u}-\eqref{i.c.}: non-negativity, upper bound preserving and limit behavior describing.
\begin{itemize}
\item {\bf Bound preserving.} \\
With the non-negative initial data satisfying uniform upper bounds in $\epsilon$, i.e.
\begin{align}
	0 \leq u_\epsilon^0(x) \leq \|u_\epsilon^0\|_{L^{\infty}(\mathbb{R}^d)}, \quad 0 \leq v_\epsilon^0(x) \leq \|v_\epsilon^0\|_{L^{\infty}(\mathbb{R}^d)}, \quad 0 \leq p_\epsilon^0(x) \leq 1, \label{i.c.bounds}
\end{align}
there is a weak solution of the system \eqref{eq:u}-\eqref{i.c.} satisfying for all $x$ and $t$
\begin{align}
	0 \leq u_\epsilon \leq \|u_\epsilon^0\|_{L^{\infty}(\mathbb{R}^d)}, \quad 0 \leq v_\epsilon \leq \|v_\epsilon^0\|_{L^{\infty}(\mathbb{R}^d)}, \quad 0 \leq p_\epsilon \leq 1. \nonumber
\end{align}
\item {\bf The limit behavior of the system \eqref{eq:u}-\eqref{i.c.} is described by the Stefan problem.} \\
Under the assumption \eqref{i.c.bounds}, one expects to find the strong limits 
$u = u(x, t) \geq 0, v = v(x, t) \geq 0$ and $0 \leq p = p(x, t) \leq 1$ of the solutions $u_\epsilon, v_\epsilon$ and $p_\epsilon$ as $\epsilon$ tends to $0$ in the system \eqref{eq:u}-\eqref{i.c.}.
This implies that the limits $u, v$ and $p$ satisfy
\begin{align}
	\begin{array}{ll}
		v = 0, ~~ p = 1, &\text{if~~} u > 0,  \\
		u = 0, ~~ p = 0, &\text{if~~} v > 0,  \\
		p \in [0, 1],  &\text{if~~} u = v = 0.
	\end{array}
	\label{limit_solution}
\end{align}
Let $w_\epsilon = u_\epsilon - v_\epsilon +  \lambda p_\epsilon$ and $w = u - v + \lambda p$, 
then the limit $w$ of $w_\epsilon$ solves the problem 
\begin{align}
	\partial_t w - \Delta B(w) = 0, ~~x \in \mathbb{R}^n, ~~ t>0, \label{eq:stefan}
\end{align}
where $B(\cdot)$ is a nonlinear diffusion function defined by
\begin{align}
	B(w) := \left\{
	\begin{array}{ll}
		d_2 w, &\text{if~~} w < 0, \\
		0, &\text{if~~} 0 \leq w \leq \lambda, \\
		d_1(w - \lambda), &\text{if~~} w > \lambda.
	\end{array}
	\right.
	\nonumber 
\end{align}
The equation \eqref{eq:stefan} is called the Stefan problem with $\lambda$ representing the latent heat. 
$w$ in the Stefan problem \eqref{eq:stefan} is called the enthalpy function of a substance, a thermodynamic state function that describes the total energy content of a system at a constant pressure, and $B(w)$ is referred to as the temperature of the substance.  
The expression \eqref{eq:stefan} describes the relationship between the temperature and enthalpy of the substance.
\end{itemize}

As the limit of the fast reaction--diffusion system \eqref{eq:u}-\eqref{i.c.} when $\epsilon$ tends to 0, the Stefan problem portrays the heat transfer process in the material undergoing a phase change. 
In the heat transfer process, the movement velocity of the interface between different phases of the substance is inscribed by the Stefan condition \cite{meirmanov2011stefan,fernandez1979generalized}.
The Stefan problem is named after the Austrian physicist Josef Stefan, who first proposed the problem in the 19th century \cite{stefan1871gleichgewicht}.  
And such a problem has wild applications in various fields, such as materials science, thermodynamics, and geophysics. 
It is commonly used to study the solidification and melting of metals, the freezing and thawing of soil, and the behavior of sea ice. 
Nowadays the Stefan problem has received great attentions in the ice accumulation problem on an aircraft and inspired by this, an enthalpy-based modeling framework to quantify uncertainty in Stefan problems with an injection boundary is developed in \cite{zhang2024uncertainty}.  

In terms of numerics, many efforts have been devoted to solving general reaction--diffusion equations and the Stefan problem respectively.
With respect to a reaction--diffusion equation, various numerical methods have been developed to approximate the solutions. 
Among them, the time-stepping methods mainly include explicit schemes, semi-implicit schemes \cite{madzvamuse2006time}, fully implicit schemes \cite{madzvamuse2014fully}, explicit-implicit schemes \cite{diele2017numerical,koto2008imex,huang2019high,wang2018third} and exponential time difference schemes \cite{khaliq2009smoothing,cox2002exponential,kassam2005fourth}, etc. 
Regarding the numerical discretization in space, the main focus is on the finite difference method, the finite element method \cite{madzvamuse2006time}, and the non-standard finite difference method \cite{anguelov2001contributions,mickens2020nonstandard}, etc. 
With respect to the Stefan problem, it is a challenging problem to solve it either analytically or numerically due to the moving interface, singularities and nonlinearities. 
Wherein accurately capturing the movement of the sharp interface between different phases has been the focus and main difficulty. 
The study of numerical methods for approximating the solution and capturing the sharp interface has attracted extensive attention from many researchers.
Depending on the different ways of tracking the sharp interface, numerical simulation methods of the Stefan problem are broadly categorized  into two directions.
One is to discretize the moving boundary problem directly based on the interface's movement velocity inscribed by the Stefan condition, such as fixed-grid finite difference methods \cite{kutluay1997numerical,caldwell2004numerical} and moving-grid finite difference methods \cite{kutluay1997numerical,beckett2001moving}. 
The other approach is to capture the interface's movement process implicitly by or based on enthalpy methods \cite{caldwell2004numerical,date1992novel,beckett2001moving,esen2004numerical,mackenzie2000numerical,tong2012smoothed,chen1997simple,zhang2024uncertainty}. 

It is a challenging problem to find a class of efficient numerical methods that can simulate both such a fast reaction--diffusion system \eqref{eq:u}-\eqref{i.c.} and its limit behavior which is described by the Stefan problem \eqref{eq:stefan} as $\epsilon$ tends to 0 in spite of the existing works in either direction. 
To be specific, the difficulty lies in several factors, including non-negativity, upper-bound preserving, stiffness, nonlinearity, parameter sensitivity, asymptotic preserving, etc.  
Therefore we aim to develop a class of efficient schemes for this type of fast reaction--diffusion system, which can not only overcome the stiffness and nonlinearity and have the non-negativity preserving and bound preserving properties on the numerical side based on the physical meaning and analytical properties of the continuum model, but also capture the limit behavior of the Stefan problem when $\epsilon$ in the numerical scheme tends to 0.

In this work, we will adopt a semi-implicit scheme, that is of first-order accuracy in time. 
Such a scheme not only satisfies the non-negativity and upper-bound preserving properties at the discrete level, but also has the stability results.
And it can accurately capture the interface propagation even with extremely fast reactions. 
We will also construct a semi-implicit Runge--Kutta scheme which is second-order accurate in time following the methodology presented in \cite{boscarino2016high}.
In addition, we perform numerical simulations of concentrations of the substances in chemical reactions, predictions of enthalpy function's behaviors in the Stefan problem and capture of sharp interfaces, thus showing strong promise for practical simulations of realistic scientific problems and providing much understanding of the nature of the mathematics behind the problem of aircraft ice accumulation. 

The rest of the paper is organized as follows. 
In section 2, we first consider the semi-implicit treatment to the fast reaction--diffusion system \eqref{eq:u}-\eqref{i.c.} which is of first-order accuracy in time, and prove that such a semi-discrete semi-implicit scheme satisfies the non-negativity preserving, bound preserving properties. 
Moreover, $L^2$ stability results of our scheme and the stability results of the corresponding scheme to the linearized system of \eqref{eq:u}-\eqref{i.c.} are given. 
Then error estimates of the scheme about the reactivity coefficients are provided.
At last, we construct a fully-discrete semi-implicit finite difference scheme.  
A semi-implicit Runge--Kutta scheme which is second-order accurate in time is outlined in section 3.
In section 4 we verify the properties of our numerical methods with numerous test examples, such as numerical accuracy, non-negativity and bound preserving, and the capturing of the sharp interface with various $\epsilon$.
Additionally, specified numerical tests are carried out with care to make simulations of the dynamics of the substances and substrate in chemical reactions and to demonstrate the processes of melting or solidification.
Some concluding remarks are made in section 5.

\section{Semi-implicit scheme and analysis} \label{sec: si}
In this section, to avoid the use of nonlinear solvers and improve the stability, we apply a semi-implicit strategy on the fast reaction--diffusion system \eqref{eq:u}-\eqref{i.c.} and propose the semi-implicit numerical schemes with the first-order convergence in time. 
And the numerical analysis of our scheme is presented in the following part.

\subsection{Semi-implicit treatment}
\label{subsec: semi}
To avoid solving nonlinear equations required by a fully implicit scheme, we focus on the semi-implicit strategy. 
When discretizing the equations \eqref{eq:u}-\eqref{eq:p} in time, to improve the stability, we treat the diffusion terms in equations \eqref{eq:u}-\eqref{eq:v} implicitly.
Then for the sake of computational simplicity and decoupling of solutions, $u$ in equation \eqref{eq:u}, $v$ in equation \eqref{eq:v}  and $p$ in equation \eqref{eq:p} are treated implicitly while the latest values are used for the other variables in each equation. 
Hence when discretizing the equations \eqref{eq:u}-\eqref{eq:p} in time, we adopt the following semi-implicit treatment to obtain the semi-discrete scheme
\begin{align}
	\frac{p^{n+1} - p^{n}}{\tau} &= \frac{1}{\epsilon} \left[\left(1 - p^{n+1}\right) u^n - v^n p^{n+1}\right], \label{semi:p} \\
	\frac{u^{n+1} - u^{n}}{\tau} - d_1 \Delta u^{n+1} &= -\frac{1}{\epsilon} u^{n+1} \left[v^n + \lambda \left(1 - p^{n+1}\right)\right],  \label{semi:u} \\
	\frac{v^{n+1} - v^{n}}{\tau} - d_2 \Delta v^{n+1} &= -\frac{1}{\epsilon} v^{n+1} \left(u^{n+1} + \lambda p^{n+1}\right), \label{semi:v} 
\end{align}
where $u^{n}, v^{n}, p^{n}$ are the approximation solutions to equations \eqref{eq:u}-\eqref{eq:p} at $t = t_n:= n \tau$ with $\tau > 0$ being the temporal size and $n = 0, 1, \cdots$.
Here in order to simplify the expression of the formulas, we drop the index $\epsilon$, and the same applies in the rest of the paper. 
We would like to highlight that such a semi-implicit treatment results in many numerical advantages, and we start by listing some desired properties of the semi-discrete scheme \eqref{semi:p}-\eqref{semi:v} in the following.
\begin{lemma}[Bound preserving]
	If the condition \eqref{i.c.bounds} holds for the initial data, the semi-discrete scheme \eqref{semi:p}-\eqref{semi:v} satisfies the non-negativity and upper-bound preserving properties.	
	\label{lemma:bp}
\end{lemma}
\begin{proof}
	Assume that $0 \leq u^n \leq \|u_\epsilon^0\|_{L^{\infty}(\mathbb{R}^d)}$, $0 \leq v^n \leq \|v_\epsilon^0\|_{L^{\infty}(\mathbb{R}^d)}$ and $0 \leq p^n \leqslant 1$,
	\eqref{semi:p} leads that
	\begin{align}
		0 \leqslant 
		p^{n+1} = \frac{p^n + \frac{\tau}{\epsilon} u^n}{1 + \frac{\tau}{\epsilon} \left(u^n + v^n\right)} 
		\leqslant 
		\frac{1 + \frac{\tau}{\epsilon} u^n}{1 + \frac{\tau}{\epsilon} u^n + \frac{\tau}{\epsilon} v^n} 
		\leq 1. 
		\label{semi bound:p}
	\end{align}
	Rewrite \eqref{semi:u} as 
	\begin{align}
		\mathcal{L}^{n+1} u^{n+1} = u^n, \nonumber
	\end{align}
	where we define the linear operator $\mathcal{L}^{n+1}:= \left[1 + \frac{\tau}{\epsilon} \left(v^n + \lambda \left(1 - p^{n+1}\right)\right)\right]\mathcal{I} - d_1 \tau \Delta$ with $\mathcal{I}$ is the identity operator. 
	The coefficient $1 + \frac{\tau}{\epsilon} \left(v^n + \lambda \left(1 - p^{n+1}\right)\right)$ in $L$ is not smaller than 1 by the assumption $v^n \geq 0$ and the double inequality \eqref{semi bound:p} and then combined with the properties of the operator $-\Delta$, one has that all eigenvalues of the operator $\mathcal{L}^{n+1}$ are not smaller than 1.
	Hence 
	\begin{align}
		0 \leqslant 
		u^{n+1} = {\mathcal{L}^{n+1}}^{-1} u^n \leqslant u^n \leq \|u_\epsilon^0\|_{L^{\infty}(\mathbb{R}^d)}, 
		\nonumber
	\end{align}
	where ${\mathcal{L}^{n+1}}^{-1}$ is the inverse operator of $\mathcal{L}^{n+1}$.
	And the proofs for the non-negativity and upper-bound preserving properties of $v$ are similar to $u$, which we omit here.
	Hence we conclude that the semi-discrete scheme \eqref{semi:p}-\eqref{semi:v} satisfies the non-negativity and upper-bound preserving properties.
\end{proof}

\begin{lemma}[$L^2$ stability]
	If the condition \eqref{i.c.bounds} and $L^2$ stability hold for the initial data, the solutions $u^{n}$ and $v^{n}$ of the semi-discrete scheme \eqref{semi:p}-\eqref{semi:v} have the $L^2$ stability.
	\label{lemma:l2s}
\end{lemma}
\begin{proof}
	Assume that $\|u^{n}\|_{L^2\left(\mathbb{R}^d\right)}$ exists.
	Multiplying $u^{n+1}$ to \eqref{semi:u} and integrating $x$ over $\mathbb{R}^d$, one has 
	\begin{align}
		\int_{\mathbb{R}^d} \left(u^{n+1}\right)^2 \,\mathrm{d} x - \int_{\mathbb{R}^d} u^n u^{n+1} \,\mathrm{d} x - d_1 \tau \int_{\mathbb{R}^d} u^{n+1} \Delta u^{n+1} \,\mathrm{d} x = -\frac{\tau}{\epsilon} \int_{\mathbb{R}^d} \left(u^{n+1}\right)^2 \left[v^n + \lambda \left(1 - p^{n+1}\right)\right] \,\mathrm{d} x.
		\nonumber
	\end{align}
	By using Cauchy inequality, integration by parts and Lemma \ref{lemma:bp}, one gets that
	\begin{align}
		\frac{1}{2} \|u^{n+1}\|_{L^2\left(\mathbb{R}^d\right)}^2 
		- \frac{1}{2} \|u^{n}\|_{L^2\left(\mathbb{R}^d\right)}^2 
		&\leq 
		\|u^{n+1}\|_{L^2\left(\mathbb{R}^d\right)}^2 
		- \int_{\mathbb{R}^d} u^n u^{n+1} \,\mathrm{d} x \nonumber \\
		&= - d_1 \|\nabla u^{n+1}\|_{L^2\left(\mathbb{R}^d\right)}^2
		-\frac{\tau}{\epsilon} \int_{\mathbb{R}^d} \left(u^{n+1}\right)^2 \left[v^n + \lambda \left(1 - p^{n+1}\right)\right] \,\mathrm{d} x \nonumber \\
		&\leq 0, \nonumber
	\end{align}
	which leads that the solution $u^{n+1}$ of the semi-discrete scheme \eqref{semi:p}-\eqref{semi:v} has the $L^2$ stability and the same goes for $v^{n+1}$.
\end{proof}

Lemma \ref{lemma:bp} and Lemma \ref{lemma:l2s} tell us that the semi-discrete semi-implicit scheme \eqref{semi:p}-\eqref{semi:v} is not only solved without a nonlinear solver, but also maintains the properties of the solutions of the fast reaction--diffusion system \eqref{eq:u}-\eqref{i.c.}, including the non-negativity, upper-bound preserving properties and the $L^2$ stability results. When considering the fully discrete scheme, the effect of the boundary conditions will be studied as well.

In addition, we mention that at the continuum level, the linearized system which is obtained by expanding the continuum model \eqref{eq:u}-\eqref{eq:p} at the limiting solutions \eqref{limit_solution} is stable. And the linearized system is written in the following form
\begin{align}
	\frac{\partial}{\partial t} s - D \Delta s = J\left(s^*\right)s.
	\label{linear3}
\end{align}
In the linearized system \eqref{linear3}, $s = \left(u, v, p\right)^\prime$, and $s^* = \left(u^*, v^*, p^*\right)^\prime$ is assumed as the limiting solution of the original semi-linear system \eqref{eq:u}-\eqref{eq:p} as $\epsilon \to 0$, i.e. 
\begin{align}
	s^* = \left\{
	\begin{array}{ll}
		s^*_1:=\left(0, v^*, 0\right)^\prime, &\text{if~~} u = 0, v > 0, \\
		s^*_2:=\left(u^*, 0, 1\right)^\prime, &\text{if~~} u > 0, v = 0, \\
		s^*_3:=\left(0, 0, p_1\right)^\prime, ~~p_1 \in [0, 1], &\text{if~~} u = 0, v = 0, \\
	\end{array}
	\right.
	\label{limit}
\end{align}
with $u^* \neq 0$ and $v^* \neq 0$, 
$w = \left(u-u^*, v-v^*, p-p^*\right)^\prime := (u, v, p)^\prime$,
the coefficient matrix $D$ and the Jacobi matrix $J$ in \eqref{linear3} are given as
\begin{align}
	D := 
	\left[
	\begin{array}{ccc}
		d_1  & &  \\
		& d_2  & \\
		& & 0
	\end{array}
	\right],
	\nonumber
\end{align}
and 
\begin{align}
	J(s^*) = \frac{1}{\epsilon} \left[
	\begin{array}{ccc}
		-v^*-\lambda(1-p^*) & -u^* & \lambda u^* \\
		-v^*& -(u^*+\lambda p^*) & -\lambda v^*\\
		1-p^*& -p^* & -(u^*+v^*)
	\end{array}
	\right].
	\label{Jacobi}
\end{align}
Next, we show that when applying the semi-discrete semi-implicit strategy we mentioned before to the linearized system \eqref{linear3}, the stability result holds at the semi-discrete level, that is to say, such a semi-discrete semi-implicit strategy does not generate unstable spatial modes. 

\begin{lemma}[Stability near the limit behavior for the linearized setting]
	For the linearized version \eqref{linear3} of the system \eqref{eq:u}-\eqref{eq:p}, the semi-discrete semi-implicit scheme is stable near the limiting solution \eqref{limit}.
	\label{lemma:ls}
\end{lemma}
\begin{proof}	
	After applying the semi-discrete semi-implicit strategy to the linearized system \eqref{linear3}, one has 
	\begin{align}
		\frac{p^{n+1} - p^{n}}{\tau} = J(s^*)_{31} u^{n} + J(s^*)_{32} v^{n} + J(s^*)_{33} p^{n+1}, \nonumber \\
		\frac{u^{n+1} - u^{n}}{\tau} - d_1 \Delta u^{n} = J(s^*)_{11} u^{n+1} + J(s^*)_{12} v^{n} + J(s^*)_{13} p^{n+1}, \nonumber\\
		\frac{v^{n+1} - v^{n}}{\tau} - d_2 \Delta v^{n} = J(s^*)_{21} u^{n+1} + J(s^*)_{22} v^{n+1} + J(s^*)_{23} p^{n+1}, \nonumber
	\end{align}
	or equivalently,
	\begin{align}
		\frac{p^{n+1} - p^{n}}{\tau} &= \frac{1}{\epsilon} [(1 - p^{*}) u^n - p^{*} v^n - (u^{*} + v^{*}) p^{n+1}], \label{linear semi:p} \\
		\frac{u^{n+1} - u^{n}}{\tau} - d_1 \Delta u^{n} &= \frac{1}{\epsilon} \left[ - (v^{*} + \lambda (1 - p^{*})) u^{n+1} - u^{*} v^n + \lambda u^{*} p^{n+1} \right] , \label{linear semi:u} \\
		\frac{v^{n+1} - v^{n}}{\tau} - d_2 \Delta v^{n} &= \frac{1}{\epsilon} \left[ - v^{*} u^{n+1} - (u^{*} + \lambda p^{*}) v^{n+1} - \lambda v^{*} p^{n+1} \right], \label{linear semi:v} 
	\end{align}
	by using the expression \eqref{Jacobi}.
	Plugging 
	$$
	p^n = \sum_{k = 1}^{\infty} \alpha_k \lambda_k^{n} e^{i m_k x}, \quad
	u^n = \sum_{k = 1}^{\infty} \beta_k \lambda_k^{n} e^{i m_k x}, \quad
	v^n = \sum_{k = 1}^{\infty} \gamma_k \lambda_k^{n} e^{i m_k x},
	$$
	into the semi-discrete semi-implicit scheme \eqref{linear semi:p}-\eqref{linear semi:v}, one has
	\begin{align}
		\frac{\gamma_k(\lambda_k - 1)}{\tau} &= \frac{1}{\epsilon} [(1 - p^{*}) \alpha_k - p^{*} \beta_k - (u^{*} + v^{*}) \lambda_k \gamma_k], \label{err:linear semi:p} \\
		\frac{\alpha_k(\lambda_k - 1)}{\tau} - \frac{d_1 \alpha_k \lambda_k (e^{i m_k h} + e^{- i m_k h} - 2)}{h^2} &= \frac{1}{\epsilon} \left[ - (v^{*} + \lambda (1 - p^{*})) \lambda_k \alpha_k - u^{*} \beta_k + \lambda u^{*} \lambda_k \gamma_k \right] , \label{err:linear semi:u} \\
		\frac{\beta_k(\lambda_k - 1)}{\tau} - \frac{d_2 \beta_k \lambda_k (e^{i m_k h} + e^{- i m_k h} - 2)}{h^2} &= \frac{1}{\epsilon} \left[ - v^{*} \lambda_k \alpha_k - (u^{*} + \lambda p^{*}) \lambda_k \beta_k -  \lambda v^{*} \lambda_k \gamma_k \right], \label{err:linear semi:v} 
	\end{align}
	where $\lambda_k$ is the growth factor.
	Let $r = \tau/h^2$, $r_\epsilon = \tau/\epsilon$, then we can rewrite the equations \eqref{err:linear semi:p}-\eqref{err:linear semi:v} as the following form
	\begin{align}
		M_{k}^{*}
		\left[
		\begin{array}{c}
			{\alpha_k} \\
			{\beta_k} \\
			{\gamma_k}
		\end{array}
		\right] = 0, \nonumber
	\end{align}
	with 
	\begin{align}
		M_{k}^{*} := 
		\left[
		\begin{array}{ccc}
			b_1(k) + r_\epsilon (v^{*} + \lambda (1 - p^{*})) \lambda_k & r_\epsilon u^{*} & -r_\epsilon \lambda u^{*} \lambda_k \\
			r_\epsilon v^{*} \lambda_k & b_2(k) + r_\epsilon (u^{*} + \lambda p^{*}) \lambda_k & r_\epsilon \lambda  v^{*} \lambda_k \\
			- r_\epsilon (1 - p^{*}) & r_\epsilon p^{*} & \lambda_k - 1 + r_\epsilon (u^{*} + v^{*}) \lambda_k 
		\end{array}
		\right],
		\label{err:matrix}
	\end{align}
	here 
	\begin{align}
		b_1(k) &= \lambda_k - 1 - 2 d_1 r \lambda_k (\cos(m_k h) - 1), \nonumber \\
		b_2(k) &= \lambda_k - 1 - 2 d_2 r \lambda_k (\cos(m_k h) - 1). \nonumber
	\end{align}
	
	According to the expression \eqref{limit} for $s^*$, there are three scenarios.
	The first case is $s^* = s_1^*$, i.e. $u^* = 0, v^* \ne 0, p^* = 0$, then \eqref{err:matrix} becomes 
	\begin{align}
		M_{k,1}^{*} = 
		\left[
		\begin{array}{ccc}
			\lambda_k - 1 - 2 d_1 r \lambda_k (\cos(m_k h) - 1) + r_\epsilon (v^{*} + \lambda) \lambda_k & 0 & 0 \\
			r_\epsilon v^{*} \lambda_k & \lambda_k - 1 - 2 d_2 r \lambda_k (\cos(m_k h) - 1) & r_\epsilon \lambda v^{*} \lambda_k \\
			- r_\epsilon & 0 & \lambda_k - 1 + r_\epsilon v^{*} \lambda_k 
		\end{array}
		\right].
		\nonumber
	\end{align}
	To have nonzero solutions $(\alpha_k, \beta_k, \gamma_k)^\prime$, we have $\det \left(M_{k,1}^{*}\right) = 0$, 
	that is to say,
	\begin{align}
		\left[ 1 + r_\epsilon (v^{*} + \lambda) + 2 d_1 r (1 - \cos(m_k h)) \right]\lambda_k = 1, \nonumber \\
		\text{or~~} \left[ 1 + 2 d_2 r (1 - \cos(m_k h)) \right]\lambda_k = 1, \nonumber \\
		\text{or~~} (1 + r_\epsilon v^{*}) \lambda_k = 1, \nonumber
	\end{align}
	which leads that $0 < \lambda_k \leq 1$ for all nonzero solutions $(\alpha_k, \beta_k, \gamma_k)^\prime$.
	The second case is $s^* = s_2^*$, i.e. $u^* \ne 0, v^* = 0, p^* = 1$, then \eqref{err:matrix} becomes 
	\begin{align}
		&M_{k,2}^{*} = \nonumber \\
		&\left[
		\begin{array}{ccc}
			\lambda_k - 1 - 2 d_1 r \lambda_k (\cos(m_k h) - 1) & r_\epsilon u^{*} & -r_\epsilon \lambda u^{*} \lambda_k \\
			0 & \lambda_k - 1 - 2 d_2 r \lambda_k (\cos(m_k h) - 1) + r_\epsilon (u^{*} + \lambda) \lambda_k & 0 \\
			0 & r_\epsilon & \lambda_k - 1 + r_\epsilon u^{*} \lambda_k 
		\end{array}
		\right].
		\nonumber
	\end{align}
	To have nonzero solutions $(\alpha_k, \beta_k, \gamma_k)^\prime$, one has $\det \left(M_{k,2}^{*}\right) = 0$, 
	i.e.
	\begin{align}
		\left[ 1 + r_\epsilon v^{*} + 2 d_1 r (1 - \cos(m_k h)) \right]\lambda_k = 1, \nonumber \\
		\text{or~~} \left[ 1 + 2 d_2 r (1 - \cos(m_k h)) + r_\epsilon (u^{*} + \lambda) \right]\lambda_k = 1, \nonumber \\
		\text{or~~} (1 + r_\epsilon u^{*}) \lambda_k = 1, \nonumber
	\end{align}
	hence $0 < \lambda_k < 1$ for all nonzero solutions $(\alpha_k, \beta_k, \gamma_k)^\prime$. The third case is $s^* = s_3^*$, i.e. $u^* = 0, v^* = 0, p^* \in [0, 1]$, then \eqref{err:matrix} becomes 
	\begin{align}
		M_{k,3}^{*} = 
		\left[
		\begin{array}{ccc}
			\lambda_k - 1 - 2 d_1 r \lambda_k (\cos(m_k h) - 1) + r_\epsilon \lambda (1 - p^{*}) \lambda_k & 0 & 0 \\
			0 & \lambda_k - 1 - 2 d_2 r \lambda_k (\cos(m_k h) - 1) + r_\epsilon \lambda p^{*} \lambda_k & 0 \\
			- r_\epsilon (1 - p^{*}) & r_\epsilon p^{*} & \lambda_k - 1 
		\end{array}
		\right].
		\nonumber
	\end{align}
	Similarly, to have nonzero solutions $(\alpha_k, \beta_k, \gamma_k)^\prime$, one has $\det \left(M_{k,3}^{*}\right) = 0$, 
	i.e.
	\begin{align}
		\left[ 1 + r_\epsilon \lambda (1 - p^{*}) + 2 d_1 r (1 - \cos(m_k h)) \right]\lambda_k = 1, \nonumber\\
		\text{or~~}\left[ 1 + 2 d_2 r (1 - \cos(m_k h)) + r_\epsilon \lambda p^{*} \right]\lambda_k = 1, \nonumber\\
		\text{or~~}\lambda_k = 1, \nonumber
	\end{align}
	which leads that $0 < \lambda_k \leq 1$ for all nonzero solutions $(\alpha_k, \beta_k, \gamma_k)^\prime$.
	
	All in all, the semi-discrete semi-implicit scheme \eqref{linear semi:p}-\eqref{linear semi:v} preserves the stability results near the limiting solution at the semi-discrete level.
\end{proof}

Note that the estimate above is near the spatial homogeneous limiting solution \eqref{linear3}.
Next, we investigate the limiting behavior involving dynamical phase transition.
Following the fact that the limit behavior of the system \eqref{eq:u}-\eqref{i.c.} is described by the Stefan problem \eqref{eq:stefan} as $\epsilon$ tends to 0, we give the residual error estimates of the scheme \eqref{semi:p}-\eqref{semi:v} about the reactivity coefficients.

Let $w^n = u^n - v^n + \lambda p^n$, then the semi-discrete semi-implicit scheme \eqref{semi:p}-\eqref{semi:v} leads that
\begin{align}
	\frac{w^{n+1} - w^{n}}{\tau} - d_1 \Delta u^{n+1} + d_2 \Delta v^{n+1} = R^{n+1}, 
	\nonumber
\end{align}
where the residual error
\begin{align}
	R^{n+1} = -\frac{1}{\epsilon} u^{n+1} \left[v^n + \lambda \left(1 - p^{n+1}\right)\right] + \frac{1}{\epsilon} v^{n+1} (u^{n+1} + \lambda p^{n+1}) + \frac{\lambda}{\epsilon} \left[\left(1 - p^{n+1}\right) u^n - v^n p^{n+1}\right]. 
	\nonumber
\end{align}
We mention that if the condition \eqref{i.c.bounds} holds for the initial data, the semi-discrete solution $w^n = u^n - v^n + \lambda p^n$ is a good approximation to the solution of the Stefan problem \eqref{eq:stefan} as long as $\epsilon$ tends to 0 and $\tau$ is small enough.

In fact, assume that $\epsilon$ tends to 0 and the condition \eqref{i.c.bounds} holds, then for $n = 1, 2, \cdots$, we consider the expansion expressions
\begin{align}
	u^n = \sum_{k = 0}^{\infty} u^n_k \epsilon^k, \quad
	v^n = \sum_{k = 0}^{\infty} v^n_k \epsilon^k, \quad
	p^n = \sum_{k = 0}^{\infty} p^n_k \epsilon^k. \label{ee} 
\end{align}
On one hand, plugging the expressions \eqref{ee} into the residual error $R^{n+1}$, one gets $R^{n+1} = \sum_{k = -1}^{\infty} R_k^{n+1} O(\epsilon^{k})$, where 
$$
R_{-1}^{n+1} = -u_0^{n+1} \left[v_0^n + \lambda \left(1 - p_0^{n+1}\right)\right] + v_0^{n+1} (u_0^{n+1} + \lambda p_0^{n+1}) + \lambda \left[\left(1 - p_0^{n+1}\right) u_0^n - v_0^n p_0^{n+1}\right],
$$
and
$$
R_{0}^{n+1} = \sum_{i + j = 1} \left( -u_i^{n+1} v_j^n + \lambda u_i^{n+1} p_j^{n+1} + u_i^{n+1} v_j^{n+1} + \lambda v_i^{n+1} p_j^{n+1} - \lambda u_i^{n} p_j^{n+1} - \lambda v_i^{n} p_j^{n+1} \right) - \lambda u_1^{n+1} + \lambda u_1^{n}.
$$
On the other hand, plugging the expressions \eqref{ee} into the semi-discrete scheme \eqref{semi:p}-\eqref{semi:v} and comparing $O({\epsilon}^{-1})$ term in the left and right sides of the scheme, we obtain
\begin{align}
	0 &= \left(1 - p_0^{n+1}\right) u_0^{n} - v_0^{n} p_0^{n+1}, \nonumber \\
	0 &= u_0^{n+1} \left[v_0^{n} + \lambda \left(1 - p_0^{n+1}\right)\right],  \nonumber \\
	0 &= v_0^{n+1} \left(u_0^{n+1} + \lambda p_0^{n+1}\right), \nonumber
\end{align}
which leads that $R_{-1}^{n+1} = 0$.
Next comparing $O(\epsilon^0)$ term in the left and right sides of the scheme, for given $\tau$, we have
\begin{align}
	\frac{p_0^{n+1} - p_0^{n}}{\tau} &= u_1^{n} + \sum_{i + j = 1} \left(- \lambda u_i^{n} p_j^{n+1} - \lambda v_i^{n} p_j^{n+1} \right), \nonumber \\
	\frac{u_0^{n+1} - u_0^{n}}{\tau} - d_1 \Delta u_0^{n+1} &= - \lambda u_1^{n+1} + \sum_{i + j = 1} \left( -u_i^{n+1} v_j^n + \lambda u_i^{n+1} p_j^{n+1} \right), \nonumber \\
	\frac{v_0^{n+1} - v_0^{n}}{\tau} - d_2 \Delta v_0^{n+1} &= \sum_{i + j = 1} \left( u_i^{n+1} v_j^{n+1} + \lambda v_i^{n+1} p_j^{n+1} \right). \nonumber 
\end{align}
Thus 
\begin{align}
	R_{0}^{n+1} = \frac{\left(u_0^{n+1} - v_0^{n+1} + \lambda p_0^{n+1}\right) - \left(u_0^{n} - v_0^{n} + \lambda p_0^{n}\right)}{\tau} - d_1 \Delta u_0^{n+1} + d_2 \Delta v_0^{n+1}.
	\nonumber
\end{align}
Recall that the scheme \eqref{semi:p}-\eqref{semi:v} is of first-order convergence in time, then one has 
\begin{align}
	R_{0}^{n+1} &= \frac{\left(u^{n+1} - v^{n+1} + \lambda p^{n+1}\right) - \left(u^{n} - v^{n} + \lambda p^{n}\right)}{\tau} - d_1 \Delta u^{n+1} + d_2 \Delta v^{n+1} + O(\frac{\epsilon}{\tau}) + O(\epsilon)\nonumber \\
	&= \partial_t w(x, t_{n+1}) - d_1 \Delta u(x, t_{n+1}) + d_2 \Delta v(x, t_{n+1}) + O(\tau) + O(\frac{\epsilon}{\tau}) + O(\epsilon) \nonumber \\ 
	&= O(\tau), ~~ \epsilon \to 0. \nonumber
\end{align}
Let the temporal size $\tau$ tends to 0, then $R_{0}^{n+1}$ tends to 0, that is to say, the semi-discrete solution $w^n = u^n - v^n + \lambda p^n$ is a good approximation to the solution of the Stefan problem \eqref{eq:stefan} as long as $\epsilon$ tends to 0 and $\tau$ is small enough. 

\subsection{First-order fully discrete semi-implicit scheme}
Now following the semi-implicit treatment in Subsection \ref{subsec: semi}, we introduce the fully discrete semi-implicit scheme of the system \eqref{eq:u}-\eqref{eq:p} equipped with the initial condition \eqref{i.c.} and Dirichlet boundary conditions 
\begin{align}
	\gamma(x, t) = \gamma_{\partial \Omega}(t), ~~\gamma = u, v, ~~x = (x_1, \cdots, x_d)^{\prime} \in \partial \Omega, ~~t > 0, \label{Dirichlet}
\end{align}
where $\gamma_{\partial \Omega}(t)$ for $\gamma = u, v$ is some given function and the computational domain $\Omega = [a, b]^{d}, b > a,$ is considered for simplicity. 
If given the number of the spatial grid nodes $N$, the uniform mesh grid on $\Omega \times [0, T]$ with the spatial size $h = (b-a)/N$ and the temporal size $\tau$ is given as $\mathcal T_{\tau, h} =\{(x_{\boldsymbol{j}}, t_{n}) ~\big|~\boldsymbol{j} = (j_1, \cdots, j_d)^\prime, ~x_{\boldsymbol{j}} = ({x_1}_{j_1}, \cdots, {x_d}_{j_d})^\prime, ~{x_i}_{j_i} = a + j_i h, ~t_{n} = n \tau, ~h = (b-a)/N, ~j_i = 0, 1, \cdots, N, ~n = 1, 2, \cdots, ~i = 1, \cdots, d \}$.  
For any $\boldsymbol{j}$ in the interior and $n = 1, 2, \cdots$, a fully discrete semi-implicit scheme of the equations \eqref{eq:u}-\eqref{eq:p} reads as
\begin{align}
	\frac{p_{\boldsymbol{j}}^{n+1} - p_{\boldsymbol{j}}^{n}}{\tau} &= \frac{1}{\epsilon} [(1 - p_{\boldsymbol{j}}^{n+1}) u_{\boldsymbol{j}}^n - v_{\boldsymbol{j}}^n p_{\boldsymbol{j}}^{n+1}], \label{fully:p} \\
	\frac{u_{\boldsymbol{j}}^{n+1} - u_{\boldsymbol{j}}^{n}}{\tau} - d_1 \Delta_h u_{\boldsymbol{j}}^{n+1} &= -\frac{1}{\epsilon} u_{\boldsymbol{j}}^{n+1} [v_{\boldsymbol{j}}^n + \lambda (1 - p_{\boldsymbol{j}}^{n+1})],  \label{fully:u} \\
	\frac{v_{\boldsymbol{j}}^{n+1} - v_{\boldsymbol{j}}^{n}}{\tau} - d_2 \Delta_h v_{\boldsymbol{j}}^{n+1} &= -\frac{1}{\epsilon} v_{\boldsymbol{j}}^{n+1} (u_{\boldsymbol{j}}^{n+1} + \lambda p_{\boldsymbol{j}}^{n+1}), \label{fully:v} 
\end{align}
where $p_{\boldsymbol{j}}^{n}, u_{\boldsymbol{j}}^{n}$ and $v_{\boldsymbol{j}}^{n}$ are the finite difference solutions on the node $(x_{\boldsymbol{j}}, t_n)$, $\Delta_h$ is the second-order center difference operator. 
In the initial time steps, for $x_{\boldsymbol{j}} \in \Omega$, we choose
\begin{align}
	p_{\boldsymbol{j}}^0 = p(x_{\boldsymbol{j}}, 0), \quad u_{\boldsymbol{j}}^0 = u(x_{\boldsymbol{j}}, 0), \quad v_{\boldsymbol{j}}^0 = v(x_{\boldsymbol{j}}, 0). \label{ic_fully}
\end{align}
And the discrete Dirichlet boundary conditions are given as 
\begin{align}
	u(x_{\boldsymbol{j}}, t_n) = u_{\partial \Omega}(t_n), \quad
	v(x_{\boldsymbol{j}}, t_n) = v_{\partial \Omega}(t_n), \label{bc_fully}
\end{align}
where $x_{\boldsymbol{j}} \in \partial \Omega$, i.e. $\exists\ i = 1, \cdots, d$, such that $j_i = 0$ or $N$. 
Rewrite the schemes \eqref{fully:p}-\eqref{fully:v} as the following form, for any $\boldsymbol{j}$ and $n = 1, 2, \cdots$,
\begin{align}
	\left[1 + \frac{\tau}{\epsilon} (u_{\boldsymbol{j}}^n + v_{\boldsymbol{j}}^n)\right] p_{\boldsymbol{j}}^{n+1} &= p_{\boldsymbol{j}}^n + \frac{\tau}{\epsilon} u_{\boldsymbol{j}}^n, \label{fully:p1} \\
	\left[1 + \frac{\tau}{\epsilon} (v_{\boldsymbol{j}}^n + \lambda (1 - p_{\boldsymbol{j}}^{n+1})) - d_1 \tau \Delta_h \right] u_{\boldsymbol{j}}^{n+1} &= u_{\boldsymbol{j}}^n, \label{fully:u1} \\
	\left[1 + \frac{\tau}{\epsilon} (u_{\boldsymbol{j}}^{n+1} + \lambda p_{\boldsymbol{j}}^{n+1}) - d_2 \tau \Delta_h \right] v_{\boldsymbol{j}}^{n+1} &= v_{\boldsymbol{j}}^n, \label{fully:v1}
\end{align}
where $\Delta_h$ is the three- and five-point center difference operator for one- and two-dimensional case respectively.
Take $d = 1$ as an example, then the first-order fully discrete semi-implicit scheme \eqref{fully:p1}-\eqref{fully:v1} in one-dimension becomes
\begin{align}
	\left[1 + \frac{\tau}{\epsilon} (u_j^n + v_j^n)\right] p_j^{n+1} &= p_j^n + \frac{\tau}{\epsilon} u_j^n, \label{fully:p2} \\
	\left[1 + \frac{\tau}{\epsilon} (v_j^n + \lambda (1 - p_j^{n+1})) + 2 d_1 \frac{\tau}{h^2}\right] u_j^{n+1} - d_1 \frac{\tau}{h^2} u_{j+1}^{n+1} - d_1 \frac{\tau}{h^2} u_{j-1}^{n+1} &= u_j^n, \label{fully:u2} \\
	\left[1 + \frac{\tau}{\epsilon} (u_j^{n+1} + \lambda p_j^{n+1}) + 2 d_2 \frac{\tau}{h^2}\right] v_j^{n+1} - d_2 \frac{\tau}{h^2} v_{j+1}^{n+1} - d_2 \frac{\tau}{h^2} v_{j-1}^{n+1} &= v_j^n. \label{fully:v2}
\end{align}

Next, we show that like the semi-discrete scheme \eqref{semi:p}-\eqref{semi:v}, the first-order fully discrete semi-implicit scheme \eqref{fully:p1}-\eqref{fully:v1} combined with the discrete initial condition \eqref{ic_fully} and boundary conditions \eqref{bc_fully}  preserves various properties of the system \eqref{eq:u}-\eqref{eq:p}, including the non-negativity preserving property, $L^2$ estimate, and the stability result for the corresponding linearized version. 
\begin{theorem}[Bound preserving]
	\label{thm:bp}
	Consider the first-order semi-implicit scheme \eqref{fully:p1}-\eqref{fully:v1}. 
	Assume that the discrete initial condition $0 \leq u_{\boldsymbol{j}}^0 \leq C_u$, $0 \leq v_{\boldsymbol{j}}^0 \leq C_v$ and $0 \leq  p_{\boldsymbol{j}}^0 \leqslant 1$ for all ${\boldsymbol{j}}$, the discrete Dirichlet boundary condition $0 \leq u_{\partial \Omega}(t_n) \leq C_u$,   $0 \leq v_{\partial \Omega}(t_n) \leq C_v$ for $n = 1, 2, \cdots$, then one has $0 \leq u_{\boldsymbol{j}}^n \leq C_u$, $0 \leq v_{\boldsymbol{j}}^n\leq C_v$ and $0 \leq p_{\boldsymbol{j}}^n \leqslant 1$ for all $\boldsymbol{j}$. 
\end{theorem}
\begin{proof}
	Assume that $0 \leq u_{\boldsymbol{j}}^n \leq C_u$, $0 \leq v_{\boldsymbol{j}}^n \leq C_v$ and $0 \leq p_{\boldsymbol{j}}^n \leqslant 1$ for any $\boldsymbol{j}$. 
	\eqref{fully:p1} leads that
	\begin{align}
		0 \leqslant 
		p_{\boldsymbol{j}}^{n+1} = \frac{p_{\boldsymbol{j}}^n + \frac{\tau}{\epsilon} u_{\boldsymbol{j}}^n}{1 + \frac{\tau}{\epsilon} (u_{\boldsymbol{j}}^n + v_{\boldsymbol{j}}^n)} 
		\leqslant 
		\frac{1 + \frac{\tau}{\epsilon} u_{\boldsymbol{j}}^n}{1 + \frac{\tau}{\epsilon} u_{\boldsymbol{j}}^n + \frac{\tau}{\epsilon} v_{\boldsymbol{j}}^n} 
		\leq 1. \label{bound:p}
	\end{align}
	
	Assume that $u_{\boldsymbol{j}}^{n+1} > C_u$ for some ${\boldsymbol{j}}$ and $u_{\boldsymbol{j}}^{n+1}$ takes the maximum at ${\boldsymbol{j}} = {\boldsymbol{j}}_0$ with $x_{\boldsymbol{j}_0} \notin \partial \Omega$. Taking ${\boldsymbol{j}} = {\boldsymbol{j}}_0$ in \eqref{fully:u1}, we find that
	\begin{align}
		\left[1 + \frac{\tau}{\epsilon} (v_{\boldsymbol{j}_0}^n + \lambda (1 - p_{\boldsymbol{j}_0}^{n+1})) - d_1 \tau \Delta_h \right] u_{\boldsymbol{j}_0}^{n+1} &= u_{\boldsymbol{j}_0}^n.
		\nonumber
	\end{align}
	When $d = 1$, one has 
	\begin{align}
		\left[1 + \frac{\tau}{\epsilon} \left(v_{j_0}^n + \lambda \left(1 - p_{j_0}^{n+1}\right)\right)\right] u_{j_0}^{n+1} + d_1 \frac{\tau}{h^2} \left(u_{j_0}^{n+1} - u_{j_0+1}^{n+1}\right) + d_1 \frac{\tau}{h^2} \left(u_{j_0}^{n+1} - u_{j_0-1}^{n+1}\right) = u_{j_0}^n. \label{fully:u1 j0}
	\end{align}
	Then $1 + \frac{\tau}{\epsilon} (v_{j_0}^n + \lambda (1 - p_{j_0}^{n+1})) \geq 1$ for the sake of \eqref{bound:p} and the assumption $v_j^n \geq 0$ for any $j$.
	Hence the left side of \eqref{fully:u1 j0} is larger than $C_u$ while the right side $u_{j_0}^n \leq C_u$, which leads to a contradictory. Thus for Dirichlet boundary, we can conclude that $u_j^{n+1} \leq u_{j_0}^{n+1} \leq C_u$. 

	Similarly, assume that $u_j^{n+1} < 0$ for some $j$ and $u_j^{n+1}$ takes the minimum at $j = j_0$. Taking $j = j_0$ in \eqref{fully:u1}, one has the left side of \eqref{fully:u1 j0} is negative while the right side is non-negative. Thus we conclude that the semi-implicit scheme preserves non-negativity.
	
	The proofs for the non-negativity and upper-bound preserving property of $v$ and for $d = 2, \cdots,$ are similar to $u$ in one-dimension, so we omit them.
\end{proof}

Let $U^n = \left(u_{\boldsymbol{j}}^n\right)^\prime$, $\|U^n\|_2 = \left(\sum_{\boldsymbol{j}, j_i \neq 0, i = 1, \cdots, d} h \left(u_{\boldsymbol{j}}^n\right)^2\right)^{\frac{1}{2}}$ 
and the same is true of $v$. Then we have the fully discrete $L^2$ stability property for $u$ and $v$ with Dirichlet boundary condition \eqref{bc_fully}. 

\begin{theorem}[$L^2$ stability]
	For Dirichlet boundary condition \eqref{bc_fully}, one has 
	$$
	\|U^{n+1}\|^2_{2} - \|U^{n}\|^2_{2} \leq h \sum_{\boldsymbol{j}, j_i = N, i = 1, \cdots, d} \left[ 
	\left({u_{\partial \Omega}}_{\boldsymbol{j}}^{n+1}\right)^2 - \left({u_{\partial \Omega}}_{\boldsymbol{j}}^{n}\right)^2 \right]
	+ 2 d_1 \frac{\tau}{h} \sum_{\boldsymbol{j}, j_i = 0, N, i = 1, \cdots, d} \left( {u_{\partial \Omega}}_{\boldsymbol{j}}^{n+1} \right)^2,
	$$
	and 
	$$
	\|V^{n+1}\|^2_{2} - \|V^{n}\|^2_{2} \leq h \sum_{\boldsymbol{j}, j_i = N, i = 1, \cdots, d} \left[ \left({v_{\partial \Omega}}_{\boldsymbol{j}}^{n+1}\right)^2 - \left({v_{\partial \Omega}}_{\boldsymbol{j}}^{n}\right)^2 \right]
	+ 2 d_2 \frac{\tau}{h} \sum_{\boldsymbol{j}, j_i = 0, N, i = 1, \cdots, d} \left( {v_{\partial \Omega}}_{\boldsymbol{j}}^{n+1} \right)^2.
	$$
	\label{thm:l2s}
\end{theorem}
\begin{proof}
	Multiplying $h u_{\boldsymbol{j}}^{n+1}$ to \eqref{fully:u1} and summing $\boldsymbol{j}$  in the set $\mathcal{J} = \left\{ \boldsymbol{j} | j_i \neq 0, N, i = 1, \cdots, d \right\}$, one has 
	\begin{align}
		h \sum_{\boldsymbol{j} \in \mathcal{J}} \left[1 + \frac{\tau}{\epsilon} \left(v_{\boldsymbol{j}}^n + \lambda \left(1 - p_{\boldsymbol{j}}^{n+1}\right)\right) - d_1 \tau \Delta_h \right] u_{\boldsymbol{j}}^{n+1} u_{\boldsymbol{j}}^{n+1} &= h \sum_{\boldsymbol{j} \in \mathcal{J}} u_{\boldsymbol{j}}^{n+1} u_{\boldsymbol{j}}^n,
		\nonumber
	\end{align}
	which leads that
	\begin{align}
		\frac{1}{2} \|U^{n+1}\|_2^2 - \frac{1}{2} \|U^{n}\|_2^2 
		&\leq 
		\frac{1}{2} h \sum_{\boldsymbol{j}, j_i = N, i = 1, \cdots, d} \left[ \left(u_{\boldsymbol{j}}^{n+1}\right)^2 - \left(u_{\boldsymbol{j}}^{n}\right)^2 \right]
		+ h \sum_{\boldsymbol{j} \in \mathcal{J}} \left[ -\frac{\tau}{\epsilon} \left(v_{\boldsymbol{j}}^n + \lambda \left(1 - p_{\boldsymbol{j}}^{n+1}\right)\right) + d_1 \tau \Delta_h \right] u_{\boldsymbol{j}}^{n+1} u_{\boldsymbol{j}}^{n+1} \nonumber \\
		&= 
		\frac{1}{2} h \sum_{\boldsymbol{j}, j_i = N, i = 1, \cdots, d} \left[ \left(u_{\boldsymbol{j}}^{n+1}\right)^2 - \left(u_{\boldsymbol{j}}^{n}\right)^2 \right]
		+ T_r + T_d, \nonumber 
	\end{align}
	with the reaction term $T_r$ and the diffusion term $T_d$ being
	\begin{align}
		T_r &:= -\frac{\tau h}{\epsilon} \sum_{\boldsymbol{j} \in \mathcal{J}} \left(v_{\boldsymbol{j}}^n + \lambda \left(1 - p_{\boldsymbol{j}}^{n+1}\right)\right) \left(u_{\boldsymbol{j}}^{n+1}\right)^2, \nonumber \\
		T_d &:= d_1 \tau h \sum_{\boldsymbol{j} \in \mathcal{J}} u_{\boldsymbol{j}}^{n+1} \Delta_h u_{\boldsymbol{j}}^{n+1}. \nonumber
	\end{align}
	Here the reaction term $T_r \leq 0$ due to Theorem \ref{thm:bp} and the diffusion term 
	\begin{align}
		T_d &= d_1 \tau \sum_{m = 1}^d \sum_{                                                                   \boldsymbol{j} \in \mathcal{J}} u_{\boldsymbol{j}}^{n+1} \left( \nabla_{h, m} u_{\boldsymbol{j}}^{n+1} - \nabla_{h, m} u_{\boldsymbol{j}-\boldsymbol{e}                                                                               _m}^{n+1} \right) \nonumber \\
		&= - d_1 \tau h \sum_{m = 1}^d \sum_{                                                                   \boldsymbol{j}} \left( \nabla_{h, m} u_{\boldsymbol{j}}^{n+1} \right)^2 
		+  d_1 \tau \sum_{m = 1}^d \sum_{                                                                   \boldsymbol{j}, j_i = N, i = 1, \cdots, d}  u_{\boldsymbol{j}}^{n+1}  \nabla_{h, m} u_{\boldsymbol{j}-\boldsymbol{e}                                                                               _m}^{n+1} 
		-  d_1 \tau \sum_{m = 1}^d \sum_{                                                                   \boldsymbol{j}, j_i = 0, i = 1, \cdots, d}  u_{\boldsymbol{j}}^{n+1}  \nabla_{h, m} u_{\boldsymbol{j}}^{n+1},
		\nonumber
	\end{align}
	where the $m$-th component of the vector $\boldsymbol{e}_m$ is 1 and the rest is 0,  and $$\nabla_{h, m} u_{\boldsymbol{j}}^{n+1} = \frac{1}{h} \left( u_{\boldsymbol{j} + \boldsymbol{e}_m}^{n+1} - u_{\boldsymbol{j}}^{n+1} \right).$$
	Then by dropping the non-positive parts in $T_d$, one gets 
	\begin{align}
		\frac{1}{2} \|U^{n+1}\|_2^2 - \frac{1}{2} \|U^{n}\|_2^2 	
		&\leq \frac{1}{2} h \sum_{\boldsymbol{j}, j_i = N, i = 1, \cdots, d} \left[ \left(u_{\boldsymbol{j}}^{n+1}\right)^2 - \left(u_{\boldsymbol{j}}^{n}\right)^2 \right]
		+ d_1 \frac{\tau}{h} \sum_{\boldsymbol{j}, j_i = 0, N, i = 1, \cdots, d} \left( u_{\boldsymbol{j}}^{n+1} \right)^2.  
		\nonumber
	\end{align}
\end{proof}

\section{High order semi-implicit Runge--Kutta scheme} 
\label{sec: rk}
In this part, we consider a high order semi-implicit Runge--Kutta scheme which is used in the case that the stiff terms don't appear in partitioned or additive form and the stiffness is associated to some variables \cite{boscarino2016high}. 
In order to apply this idea, we rewrite the stiff system \eqref{eq:u}-\eqref{eq:p} as 
\begin{align}
	\partial_t u - d_1 \Delta \overline{u} &= - \frac{1}{\epsilon} \overline{u} \left[v + \lambda \left(1 - p\right)\right], ~~x \in \mathbb{R}^d, ~~t>0,  \label{rk eq:u} \\
	\partial_t v - d_2 \Delta \overline{v} &= - \frac{1}{\epsilon} \overline{v} \left(u + \lambda p\right), \label{rk eq:v} \\
	\partial_t p &= \frac{1}{\epsilon} \left(\overline{u} - u \overline{p} - v \overline{p}\right), \label{rk eq:p}
\end{align}
where $\overline{u} = u$, $\overline{v} = v$ and $\overline{p} = p$. 
Define the vector $y = (u, v, p)^\prime$ and $z = (\overline{u}, \overline{v}, \overline{p})^\prime$, then the equations \eqref{rk eq:u}-\eqref{rk eq:p} can be written in the following form,
\begin{align}
	\partial_t y &= \mathcal{H}(y, z), \nonumber \\
	\partial_t z &= \mathcal{H}(y, z), \nonumber
\end{align}
with the right hand side 
\begin{align}
	\mathcal{H}(y, z) := \left[
	\begin{array}{c}
		d_1 \Delta \overline{u} - \frac{1}{\epsilon} \overline{u} \left[v + \lambda (1 - p)\right] \\
		d_2 \Delta \overline{v}(t) - \frac{1}{\epsilon} \overline{v} (u + \lambda p) \\
		\frac{1}{\epsilon} \left(\overline{u} - u \overline{p} - v \overline{p}\right) \\
	\end{array}
	\right]. \nonumber
\end{align}
The diffusion term is treated implicitly as it induces some stiffness, and $z$ is treated implicitly while $y$ is treated explicitly.  
Set $z^n = y^n$ and compute the stage fluxes for $i = 1, \cdots, s$,  
\begin{align}
	Y_i &= y^n + \tau \sum_{j = 1}^{i - 1} \hat{a}_{ij} k_j, ~~i = 2, \cdots, s, \label{rk:y} \\
	\tilde{Z}_i &= y^n + \tau \sum_{j = 1}^{i - 1} a_{ij} k_j, ~~i = 2, \cdots, s, \label{rk:z}\\
	k_i &= \mathcal{H}\left(Y_i,  \tilde{Z}_i + \tau a_{ii} k_i \right), ~~i = 1, \cdots, s, \label{rk:k}
\end{align}
with $Y_1 = \tilde{Z}_1 = y^n$.
Then the solution $y^{n+1} = y^n + \tau \sum_{i = 1}^{s} \hat{b}_i k_i$. 
Hence combined with proper discretization in space, an high order semi-implicit Runge--Kutta scheme is proposed.

In this part, we adopt the 2-stage second-order semi-implicit Runge--Kutta scheme whose double Butcher tableau is 
\begin{equation}
	\begin{tabular}{l|c}
		$\boldsymbol{\hat{c}}$ & $\boldsymbol{\hat{A}}$\\ \hline
		& $\boldsymbol{\hat{b}}$
	\end{tabular},\qquad 
	\begin{tabular}{l|c}
		$\boldsymbol{c}$ & $\boldsymbol{A}$\\ \hline
		& $\boldsymbol{b}$
	\end{tabular},
	\nonumber
\end{equation}
with the strictly lower-triangular matrix
\[
\boldsymbol{\hat{A}} = \left[
\begin{array}{cc}
	\hat{a}_{11} & \hat{a}_{12} \\
	\hat{a}_{21} & \hat{a}_{22}
\end{array}
\right] = \left[
\begin{array}{cc}
	0 & 0 \\
	c_0 & 0
\end{array}
\right],
\]
and the lower-triangular matrix
\[
\boldsymbol{A} = \left[
\begin{array}{cc}
	\hat{a}_{11} & \hat{a}_{12} \\
	\hat{a}_{21} & \hat{a}_{22}
\end{array}
\right] = \left[
\begin{array}{cc}
	\gamma & 0 \\
	1-\gamma & \gamma
\end{array}
\right],
\]
with $\gamma = 1 - 1/\sqrt{2}$ and $c_0 = 1/(2 \gamma)$, the vectors 
\[
\boldsymbol{b} = \boldsymbol{\hat{b}} = \left[\hat{b}_1, \hat{b}_2\right]^\prime = \left[1-\gamma, \gamma \right]^\prime,
\]
and the vectors $\boldsymbol{\hat{c}} = \left[\hat{c}_1,\hat{c}_2\right]^\prime =\left[0,c_0\right]^\prime$, $\boldsymbol{c} = \left[c_1, c_2\right]^\prime = \left[\gamma, 1\right]^\prime$, which is obtained from $\hat{c}_i = \sum_{j=1}^{i-1} \hat{a}_{ij}$ and $c_i = \sum_{j=1}^i a_{ij}$.
We remark that a general $s$-stage semi-implicit Runge--Kutta scheme can also be considered. 
Next, we give the calculations of the semi-implicit Runge--Kutta scheme \eqref{rk:y}-\eqref{rk:k} if the solution $y^n$ is given.

\begin{itemize}
	\item Stage 1. \\
	The scheme \eqref{rk:y}-\eqref{rk:k} becomes
	\begin{align}
		Y_1 &= y^n, \nonumber \\
		\tilde{Z}_1 &= y^n, \nonumber \\
		k_1 &= \mathcal{H}\left(Y_1, \tilde{Z}_1 + \tau a_{11} k_1 \right). \nonumber 
	\end{align}
	Then one has 
	\begin{align}
		M_{1} k_{1} = f_{1}, \nonumber
	\end{align}
	where 
	\begin{align}
		k_{1} = \left[ 
		\begin{array}{c}
			k_{11} \\
			k_{12} \\
			k_{13}
		\end{array}
		\right],  \quad
		f_{1} 
		= \left[
		\begin{array}{c}
			- \frac{1}{\epsilon} \tilde{Z}_{11} [Y_{12} + \lambda (1 - Y_{13})] + d_1 \Delta_h \tilde{Z}_{11} \\
			- \frac{1}{\epsilon} \tilde{Z}_{12} (Y_{11} + \lambda Y_{13}) + d_2 \Delta_h  \tilde{Z}_{12} \\
			\frac{1}{\epsilon} \tilde{Z}_{11} - \frac{1}{\epsilon} (Y_{11} + Y_{12}) \tilde{Z}_{13}
		\end{array}
		\right], \nonumber
	\end{align}
	and
	\begin{align}
		M_{1}
		= \left[
		\begin{array}{ccc}
			1 + \frac{1}{\epsilon} \tau a_{11} [Y_{12} + \lambda (1 - Y_{13})] - d_1 \tau a_{11} \Delta_h & 0 & 0\\
			0 & 1 + \frac{1}{\epsilon} \tau a_{11} (Y_{11} + \lambda Y_{13}) - d_2 \tau a_{11} \Delta_h & 0\\
			\frac{1}{\epsilon} \tau a_{11} & 0 & 1 + \frac{1}{\epsilon} \tau a_{11} (Y_{11} + Y_{12})
		\end{array}
		\right], \nonumber
	\end{align}
	with $Y_1 = \left(Y_{11}, Y_{12}, Y_{13}\right)^\prime$ and $\tilde{Z}_1 = \left(\tilde{Z}_{11}, \tilde{Z}_{12}, \tilde{Z}_{13}\right)^\prime$.\\
	\item Stage 2. \\
	\begin{align}
		Y_2 &= y^n + \tau \hat{a}_{21} k_1,  \nonumber \\
		\tilde{Z}_2 &= y^n + \tau a_{21} k_1, \nonumber \\
		k_2 &= \mathcal{H} \left(Y_2, \tilde{Z}_2 + \tau a_{22} k_2\right). \nonumber
	\end{align}
	Then one has 
	\begin{align}
		M_{2} k_{2} = f_{2}, \nonumber
	\end{align}
	where 
	\begin{align}
		k_{2} = \left[ 
		\begin{array}{c}
			k_{21} \\
			k_{22} \\
			k_{23}
		\end{array}
		\right],  \quad
		f_{2} 
		= \left[
		\begin{array}{c}
			- \frac{1}{\epsilon} \tilde{Z}_{21} [Y_{22} + \lambda (1 - Y_{23})] + d_1 \Delta_h \tilde{Z}_{21} \\
			- \frac{1}{\epsilon} \tilde{Z}_{22} (Y_{21} + \lambda Y_{23}) + d_2 \Delta_h  \tilde{Z}_{22} \\
			\frac{1}{\epsilon} \tilde{Z}_{21} - \frac{1}{\epsilon} (Y_{21} + Y_{22}) \tilde{Z}_{23}
		\end{array}
		\right], \nonumber
	\end{align}
	and
	\begin{align}
		M_{2}
		= \left[
		\begin{array}{ccc}
			1 + \frac{1}{\epsilon} \tau a_{22} [Y_{22} + \lambda (1 - Y_{23})] - d_1 \tau a_{22} \Delta_h & 0 & 0\\
			0 & 1 + \frac{1}{\epsilon} \tau a_{22} (Y_{21} + \lambda Y_{23}) - d_2 \tau a_{22} \Delta_h & 0\\
			\frac{1}{\epsilon} \tau a_{22} & 0 & 1 + \frac{1}{\epsilon} \tau a_{22} (Y_{21} + Y_{22})
		\end{array}
		\right], \nonumber
	\end{align}
	with $Y_2 = \left(Y_{21}, Y_{22}, Y_{23}\right)^\prime$ and $\tilde{Z}_2 = \left(\tilde{Z}_{21}, \tilde{Z}_{22}, \tilde{Z}_{23}\right)^\prime$.\\
	
	Hence the solution of $y$ at $t_{n+1}$ is given by
	\begin{align}
		y^{n+1} = y^n + \tau \left(\hat{b}_1 k_1 + \hat{b}_2 k_2 \right).  \nonumber
	\end{align}
\end{itemize}

We mention again that such a semi-implicit Runge--Kutta scheme in Section \ref{sec: rk} is practically of second-order accuracy in both time and space and numerical tests on the convergence order in time will be carried out in next Section.

\section{Numerical tests} 
In this section, we provide various numerical tests to verify the accuracy of the first- and second-order schemes in Section \ref{sec: si} and Section \ref{sec: rk}. 
Besides the order of accuracy, we also observe the non-negativity and bound preserving property, and the capture of the sharp interface with various $\epsilon$ on the numerical side. 
At last we simulate the dynamics of the substances and the substrate, the liquid solidification process, etc.
\subsection{Convergence tests} 
The errors of numerical solutions $\{\gamma^n_{\boldsymbol{j}}\}$ for $\gamma = u, v, p, w$ are computed as follows, 
\begin{align}
	&{e_\gamma}_{\tau,h}^{\infty} := \max_{\boldsymbol{j}} |\gamma^n_{\boldsymbol{j}} - \gamma^{\text{ref}}(\boldsymbol{x}_{\boldsymbol{j}}, t_n)|, \nonumber\\
	&{e_\gamma}_{\tau,h}^{1} := h \sum_{\boldsymbol{j}} |\gamma^n_{\boldsymbol{j}} - \gamma^{\text{ref}}(\boldsymbol{x}_{\boldsymbol{j}}, t_n)|, \nonumber \\
	&{e_\gamma}_{\tau,h}^{2} := \left[h \sum_{\boldsymbol{j}} |\gamma^n_{\boldsymbol{j}} - \gamma^{\text{ref}}(\boldsymbol{x}_{\boldsymbol{j}}, t_n)|^2\right]^{\frac{1}{2}}, \nonumber
\end{align}
where $j=(j_1, \cdots, j_d)^\prime$, $\boldsymbol{x}_{\boldsymbol{j}} = ({x_1}_{j_1}, \cdots, {x_d}_{j_d})^\prime$, $h$ is the spatial mesh size, $\tau$ is the temporal mesh size, and $\gamma^{\text{ref}}$ is the reference solution which is obtained by small enough $h$ and $\tau$.

\subsubsection{Convergence test for the first-order semi-implicit scheme}
{\bf Case 1: a one-dimensional test with $\epsilon \sim O(1)$.} Consider the fast reaction--diffusion system \eqref{eq:u}-\eqref{eq:p}. We first apply the first-order semi-implicit scheme to such model with the coefficients $\epsilon = 1$ and the initial condition \eqref{i.c.} and the Dirichlet boundary condition \eqref{Dirichlet} given in the following form
\begin{align}
	&v(x, 0) = \cos\left(\frac{\pi x}{2 L}\right) + \epsilon_0, \quad u(x, 0) = 1 - v + \epsilon_0, \quad p(x, 0) = \frac{x}{2L} + \frac{1}{2}, \quad \epsilon_0 = 10^{-8}, \quad x \in [-L, L], \nonumber\\
	&u(-L, t) = u(L, t) = 1, \quad v(-L, t) = v(L, t) = \epsilon_0. \nonumber
\end{align}

In the first test, we verify the accuracy order of our scheme in both space and time. The computation domain is taken as $[-L, L], L = 1$, the results of the first-order convergence in time at $t = 0.1$ are shown in Figure \ref{fig:con_eps0_1D} (left) and Table \ref{tab:con_eps0_1D_t}, where we take the time step size $\tau = 2^{-j} \tau_0, j = 0, 1, \cdots, 9, 10$ with $\tau_0 = 0.002$ and the spatial mesh size $h = 1.25 \times 10^{-2}$. In this test, the reference solution is obtained by taking $h = 1.25 \times 10^{-2}$ and $\tau = 2^{-13} \tau_0$. Then Figure \ref{fig:con_eps0_1D} (Right) and Table \ref{tab:con_eps0_1D_x} show the second order convergence in space, where we take the uniform mesh size $h = 2^{-j} h_0, j = 0, 1, \cdots, 5, 6$ with $h_0 = 4 \times 10^{-2}$ and $\tau = 6.25 \times 10^{-3}$. In this test, the solution on mesh with mesh size $h = 2^{-8} h_0,  \tau = 6.25 \times 10^{-3}$ is regarded as the reference solution.
\begin{figure}[htp]
	\centering
	\includegraphics[width=8cm,height=7cm]{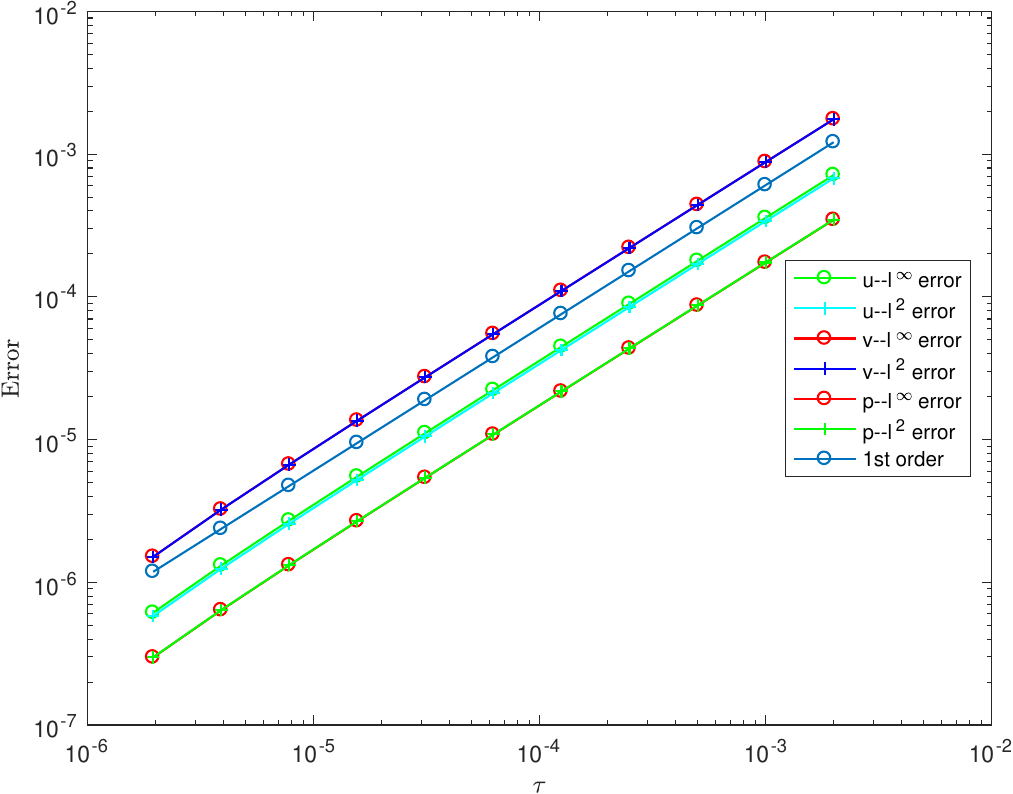}
	\includegraphics[width=8cm,height=7cm]{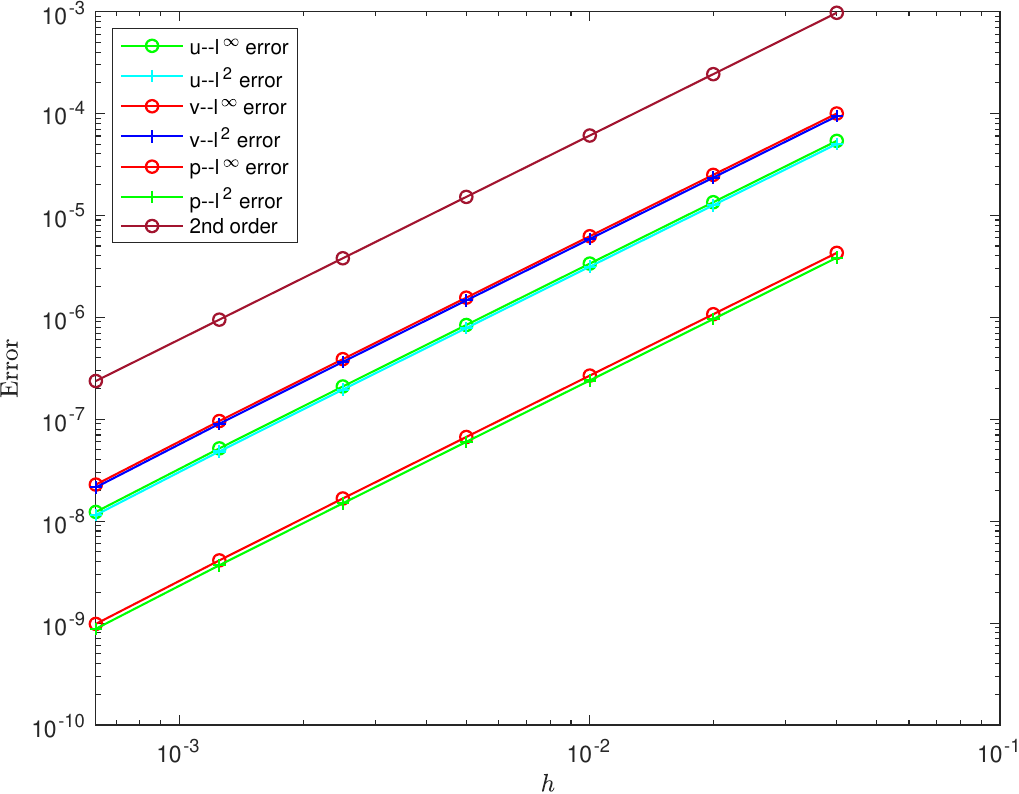}
	\caption{Left: First-order semi-implicit scheme: accuracy order check in $\tau$ with $\epsilon = 1$ and $h = 1.25 \times 10^{-2}, \tau = 2^{-j} \tau_0, j = 0, 1, \cdots, 9, 10$. Right: First-order semi-implicit scheme: accuracy order check in $h$ with $\epsilon = 1$ and $\tau = 6.25 \times 10^{-3}, h = 2^{-j} h_0, j = 0, 1, \cdots, 5, 6$.}
	\label{fig:con_eps0_1D}
\end{figure}

\begin{table}[htbp]
	\centering
	\caption{First-order semi-implicit scheme: accuracy order check in $\tau$ with $\epsilon = 1$, $h = 1.25 \times 10^{-2}$}
	\label{tab:con_eps0_1D_t}
	\begin{tabular}{|p{1.8cm}|p{1.8cm}|p{1.8cm}|p{1.8cm}|p{1.8cm}|}
		\hline
		$\tau$ & ${e_u}_{\tau,h}^{\infty}$ & order & ${e_u}_{\tau,h}^{2}$ & order \\
		\hline
		1.9531e-05 &  6.1239e-07 &         &  5.8105e-07 &    \\
		3.9063e-05 &  1.3123e-06 &  1.0995 &  1.2451e-06 &  1.0995\\
		7.8125e-05 &  2.7120e-06 &  1.0473 &  2.5732e-06 &  1.0473\\
		1.5625e-04 &  5.5114e-06 &  1.0231 &  5.2294e-06 &  1.0231\\
		3.1250e-04 &  1.1110e-05 &  1.0114 &  1.0542e-05 &  1.0114\\
		6.2500e-04 &  2.2306e-05 &  1.0056 &  2.1165e-05 &  1.0056\\
		1.2500e-03 &  4.4696e-05 &  1.0027 &  4.2409e-05 &  1.0027\\
		2.5000e-03 &  8.9460e-05 &  1.0011 &  8.4884e-05 &  1.0011\\
		5.0000e-03 &  1.7893e-04 &  1.0001 &  1.6979e-04 &  1.0001\\
		1.0000e-02 &  3.5766e-04 &  0.99917 &  3.3939e-04 &  0.99923\\
		1.0000e-02 &  7.1423e-04 &  0.99780 &  6.7781e-04 &  0.99793\\
		\hline
	\end{tabular}
\end{table}

\begin{table}[htbp]
	\centering
	\caption{First-order semi-implicit scheme: accuracy order check in $h$ with $\epsilon = 1$, $\tau = 6.25 \times 10^{-3}$}
	\label{tab:con_eps0_1D_x}
	\begin{tabular}{|p{1.8cm}|p{1.8cm}|p{1.8cm}|p{1.8cm}|p{1.8cm}|}
		\hline
		$h$ & ${e_u}_{\tau,h}^{\infty}$ & order & ${e_u}_{\tau,h}^{2}$ & order \\
		\hline
		6.2500e-04 &  1.2312e-08 &         &  1.1487e-08 & 	\\
		1.2500e-03 &  5.1710e-08 &  2.0704 &  4.8248e-08 &  2.0704\\
		2.5000e-03 &  2.0930e-07 &  2.0171 &  1.9529e-07 &  2.0171\\
		5.0000e-03 &  8.3967e-07 &  2.0042 &  7.8345e-07 &  2.0042\\
		1.0000e-02 &  3.3611e-06 &  2.0010 &  3.1361e-06 &  2.0010\\
		2.0000e-02 &  1.3444e-05 &  1.9999 &  1.2546e-06 &  2.0002\\
		4.0000e-02 &  5.3758e-05 &  1.9995 &  5.0180e-05 &  1.9999 \\
		\hline
	\end{tabular}
\end{table}

{\bf Case 2: a one-dimensional test with $\epsilon \sim O(10^{-2})$.}
We next consider the coefficients $\epsilon = 10^{-2}$. 
The homogeneous Neumann boundary condition is considered at each boundary, and the initial condition \eqref{i.c.} is given in the following form
\begin{align}
	u(x, 0) = 
	\left\{
	\begin{array}{ll}
		0, & x < 0, \\
		\frac{\theta}{d_1}, & x > 0,
	\end{array}
	\right.  \quad
	v(x, 0) = 
	\left\{
	\begin{array}{ll}
		\frac{S_t}{d_2}, & x < 0, \\
		0, & x > 0,
	\end{array}
	\right. \quad
	p(x, 0) = 
	\left\{
	\begin{array}{ll}
		0, & x < 0, \\
		1, & x > 0,
	\end{array}
	\right.
	\nonumber
\end{align} 
where $\theta = 0.05, S_t = 0.25, d_1 = 1, d_2 = 2$.

In this case, we verify the accuracy order of our scheme in time. The computation domain is taken as $[-L, L], L = 1$, the results of the first-order convergence in time at $t = 0.5$ are shown in Figure \ref{fig:con_epspt01_1D} (left), where we take the time step size $\tau = 2^{-j} \tau_0, j = 0, 1, \cdots, 9, 10$ with $\tau_0 = 0.01$ and the spatial mesh size $h = 1.25 \times 10^{-2}$. In this test, the reference solution is obtained by taking $h = 1.25 \times 10^{-2}$ and $\tau = 2^{-13} \tau_0$.

\begin{figure}[htp]
	\centering
	\includegraphics[width=8cm,height=7cm]{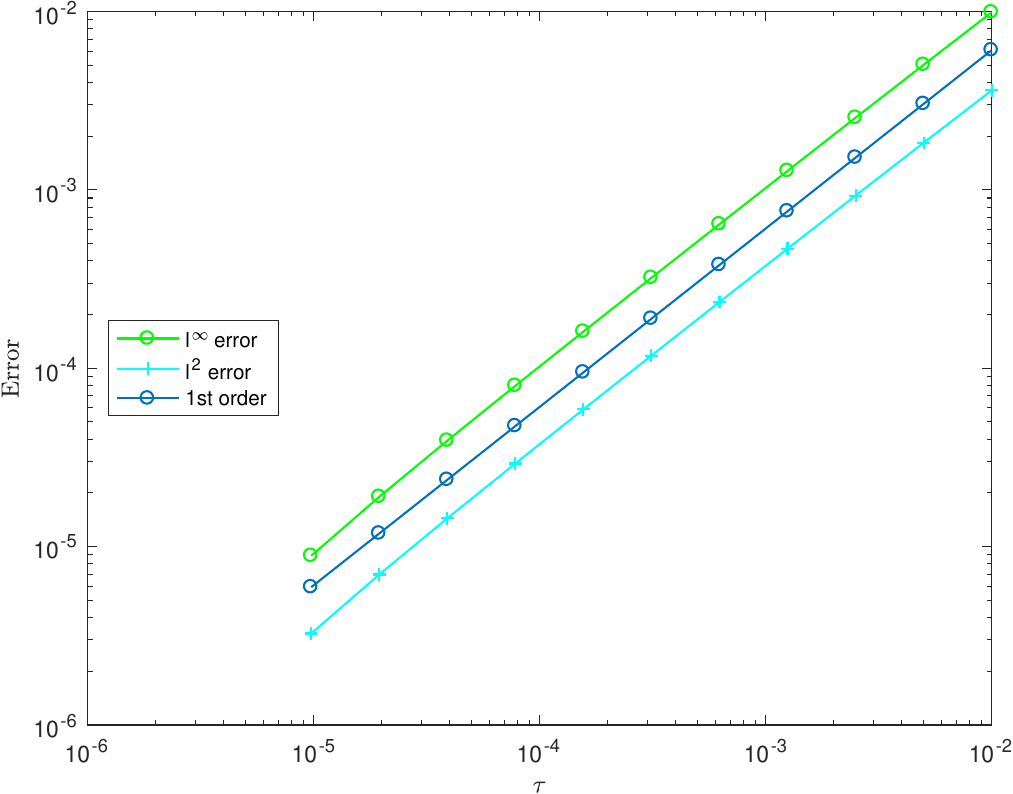}
	\includegraphics[width=8cm,height=7cm]{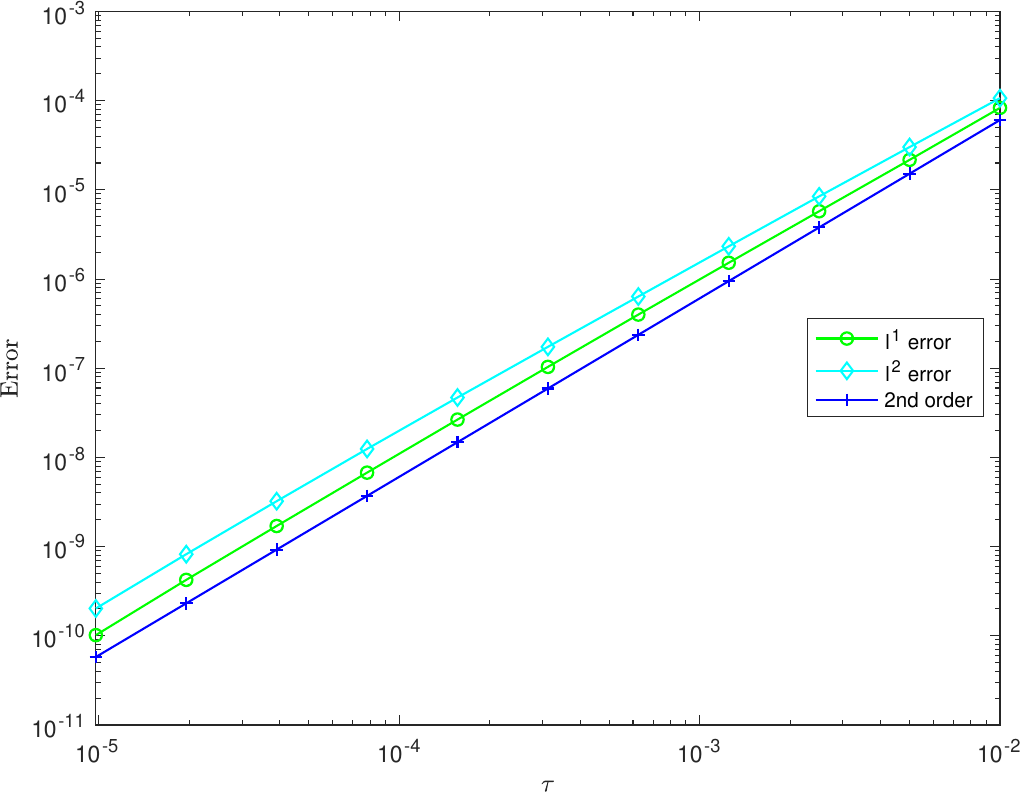}
	\caption{Left: First-order semi-implicit scheme: accuracy order check in $\tau$ with $\epsilon = 10^{-2}$ and $h = 1.25 \times 10^{-2}, \tau = 2^{-j} \tau_0, j = 0, 1, \cdots, 9, 10$. Right: Second-order semi-implicit Runge--Kutta scheme: accuracy order check in $\tau$ with $\epsilon = 10^{-2}$ and $h = 6.25 \times 10^{-3}, \tau = 2^{-j} \tau_0, j = 0, 1, \cdots, 9, 10$.}
	\label{fig:con_epspt01_1D}
\end{figure}

\subsubsection{Convergence test for the second-order semi-implicit Runge--Kutta scheme}
Consider the coefficients $\epsilon = 10^{-2}$. 
The initial and boundary condition is given in the following form
\begin{equation}
	\begin{aligned}
		&u_\epsilon(x, 0) = \frac{\theta}{d_1}, \quad
		v_\epsilon(x, 0) = 0, \quad
		p_\epsilon(x, 0) = 1,  \\
		&u_\epsilon(-L, t) = 0, \quad
		v_\epsilon(-L, t) = \frac{S_t}{d_2}, \quad
		p_\epsilon(-L, t) = 0, \\
		&u_\epsilon(L, t) = \frac{\theta}{d_1}, \quad
		v_\epsilon(L, t) = 0, \quad
		p_\epsilon(L, t) = 1, 
		\nonumber
	\end{aligned}
\end{equation}
where $d_1 = 1$, $d_2 = 2$, $\lambda = 1$, $\theta = 0.05$, $S_t = 0.25$, $L = 1$.

In this case, we verify the accuracy order in time of the semi-implicit Runge--Kutta scheme in Section \ref{sec: rk}. The computation domain is taken as $[-L, L]$, the results of the second-order convergence in time at $t = 0.5$ are shown in Figure \ref{fig:con_epspt01_1D}(Right), where we take the time step size $\tau = 2^{-j} \tau_0, j = 0, 1, \cdots, 9, 10$ with $\tau_0 = 0.01$ and the spatial mesh size $h = 6.25 \times 10^{-3}$. In this test, the reference solution is obtained by taking $h = 6.25 \times 10^{-3}$ and $\tau = 2^{-13} \tau_0$.

\subsection{Limit behavior}
\subsubsection{One-dimensional case}
In order to demonstrate the limit behavior of the fast reaction--diffusion system \eqref{eq:u}-\eqref{eq:p}, we consider the following model with Dirichlet boundary conditions on the domain $[-L, L] \times [0, T]$,
\begin{equation}\label{test:limit}
	\begin{aligned}
		&\partial_t u_\epsilon - d_1 \Delta u_\epsilon = -\frac{1}{\epsilon} u_\epsilon [v_\epsilon + \lambda (1 - p_\epsilon)],  \\
		&\partial_t v_\epsilon - d_2 \Delta v_\epsilon = -\frac{1}{\epsilon} v_\epsilon (u_\epsilon + \lambda p_\epsilon),  \\
		&\partial_t p_\epsilon =  \frac{1}{\epsilon} [(1 - p_\epsilon) u_\epsilon - v_\epsilon p_\epsilon],  \\
		&u_\epsilon(x, 0) = \frac{\theta}{d_1}, \quad
		v_\epsilon(x, 0) = 0, \quad 
		p_\epsilon(x, 0) = 1,  \\
		&u_\epsilon(-L, t) = 0, \quad 
		v_\epsilon(-L, t) = \frac{S_t}{d_2}, \quad
		p_\epsilon(-L, t) = 0, \\
		&u_\epsilon(L, t) = \frac{\theta}{d_1}, \quad
		v_\epsilon(L, t) = 0, \quad 
		p_\epsilon(L, t) = 1, 
	\end{aligned}
\end{equation}
where $d_1 = 1$, $d_2 = 2$, $\lambda = 1$, $\theta = 0.05$, $S_t = 0.25$, $L = 1$, $T = 1$.

In this test, we first investigate the dynamic behavior of the solutions to the system \eqref{test:limit}. We first take $\epsilon = 10^{-2}$ and the mesh size $h = 0.0125$ and $\tau = 5 \times 10^{-6}$. Figure \ref{plot: eppt01_limit} shows the results of $u, v, p$ and the enthalpy function $w$ at time $t$ from $0$ to $1$. 
We can numerically check the positivity and bound preserving properties, which agrees with the numerical analysis in Section \ref{sec: si}. Furthermore, as the strength of the reaction is much greater than that of the diffusion process, we can find that the increase or decrease of $u$ is consistent with $p$, and the opposite is true for $v$.
Next we retake $\epsilon = 10^{-4}$, Figure \ref{plot: eppt0001_limit} shows the corresponding results with $h = 6.25 \times 10^{-3}$ and $\tau = 5 \times 10^{-6}$. 
After comparing the dynamic behavior of the solutions with different $\epsilon$, we can conclude that the smaller the $\epsilon$ the sharper the interface. 
The sharp interface is referred to as the phase transition interface in the limit Stefan problem, whose speed is determined by the Stefan's condition. 
In this way, assume that the system is full of liquid at the initial moment, such as water, and give a freezing stimulus at the left edge of the system, then we can conclude that the liquid exothermizes and freezes, and the movement of the sharp interface moving to the right inscribes the liquid solidification process.
Then Figure \ref{plot: compare_eps} shows more details on the sharp interface with various $\epsilon$.

From Figure \ref{plot: compare_eps_implicit} we can conclude that for the first-order scheme, smaller parameters $\epsilon$ impose tighter requirements on the mesh size. Combined with the limit behavior, proposing an asymptotic preserving scheme on $\epsilon$, which avoids the use of very fine meshes, is our future consideration.

\begin{figure}[htp]
	\centering
	\includegraphics[width=5.0cm,height=5cm]{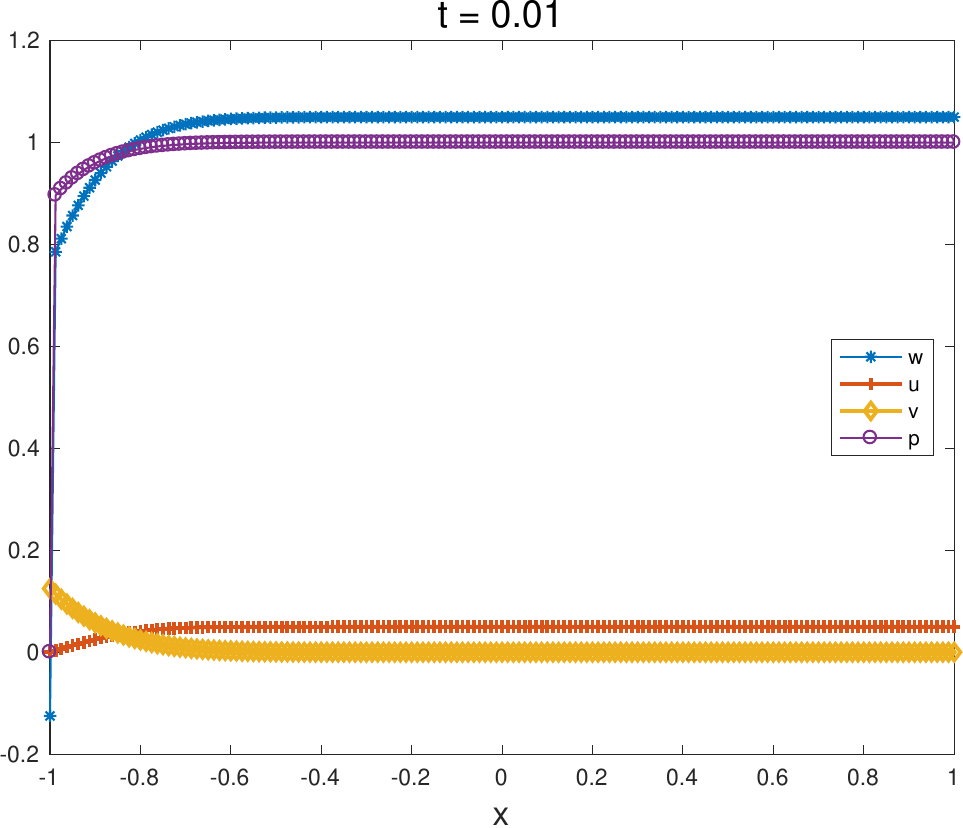}
	\includegraphics[width=5.0cm,height=5cm]{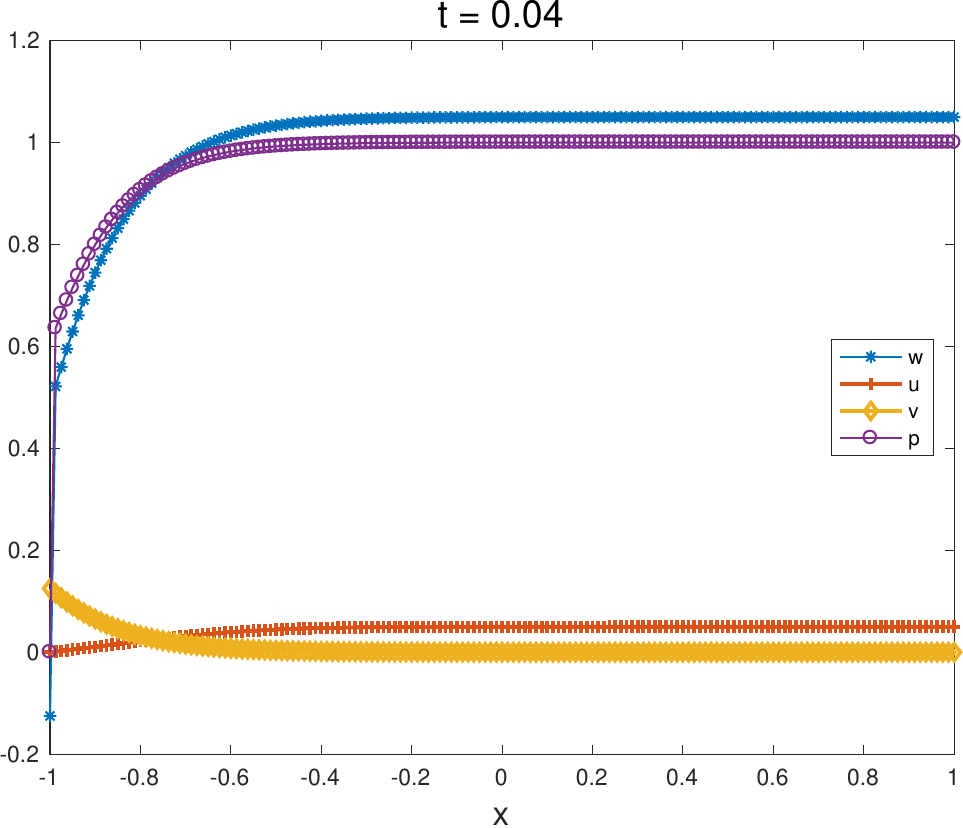}
	\includegraphics[width=5.0cm,height=5cm]{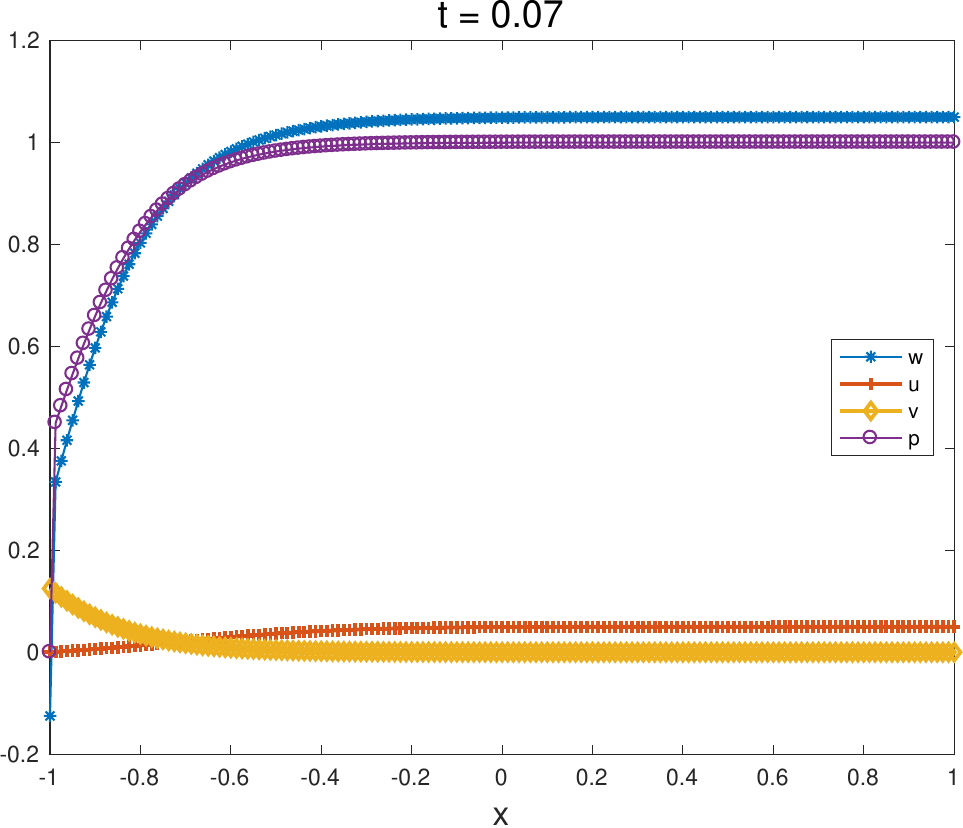}
	\includegraphics[width=5.0cm,height=5cm]{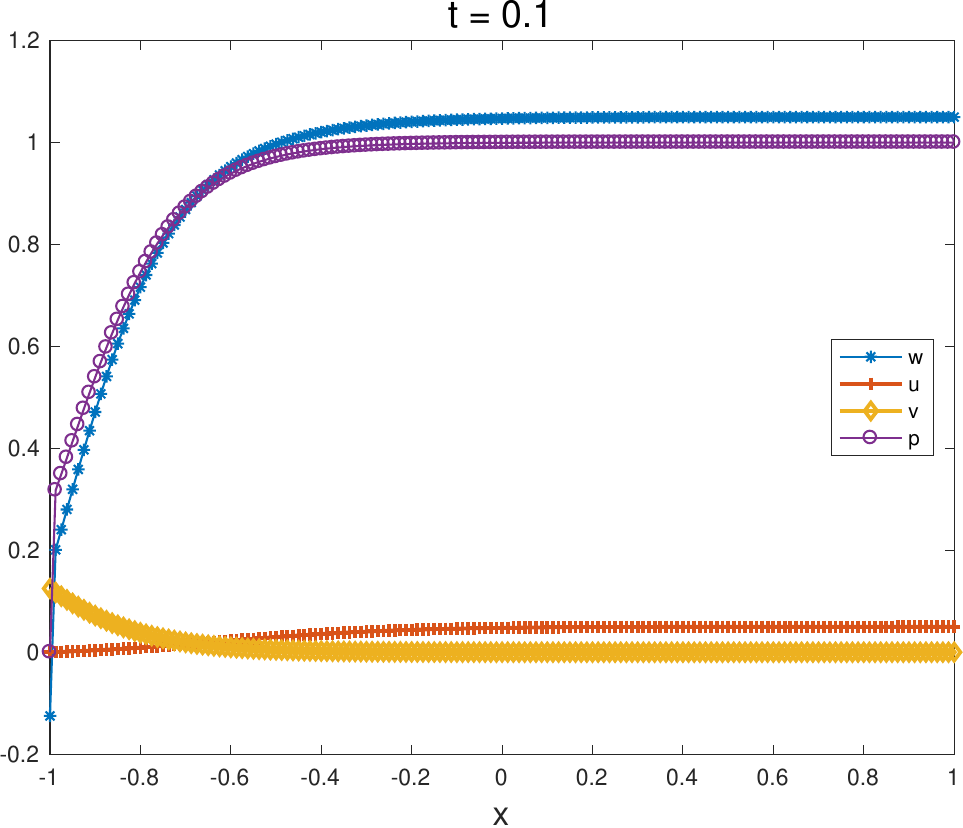}
	\includegraphics[width=5.0cm,height=5cm]{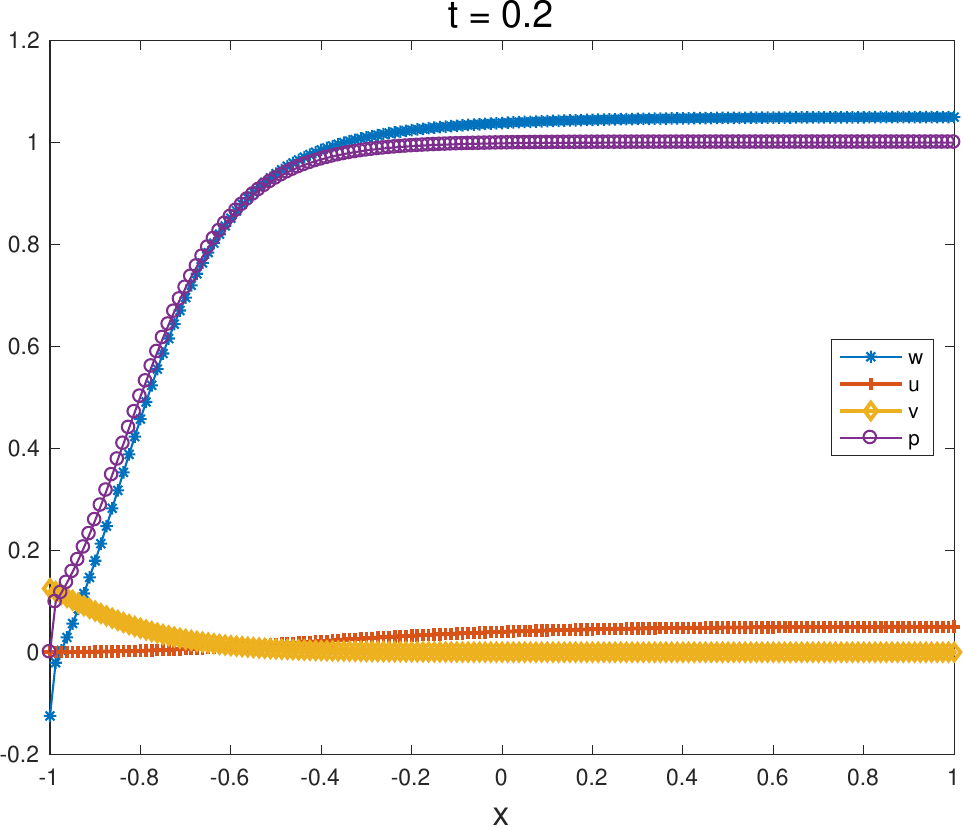}
	\includegraphics[width=5.0cm,height=5cm]{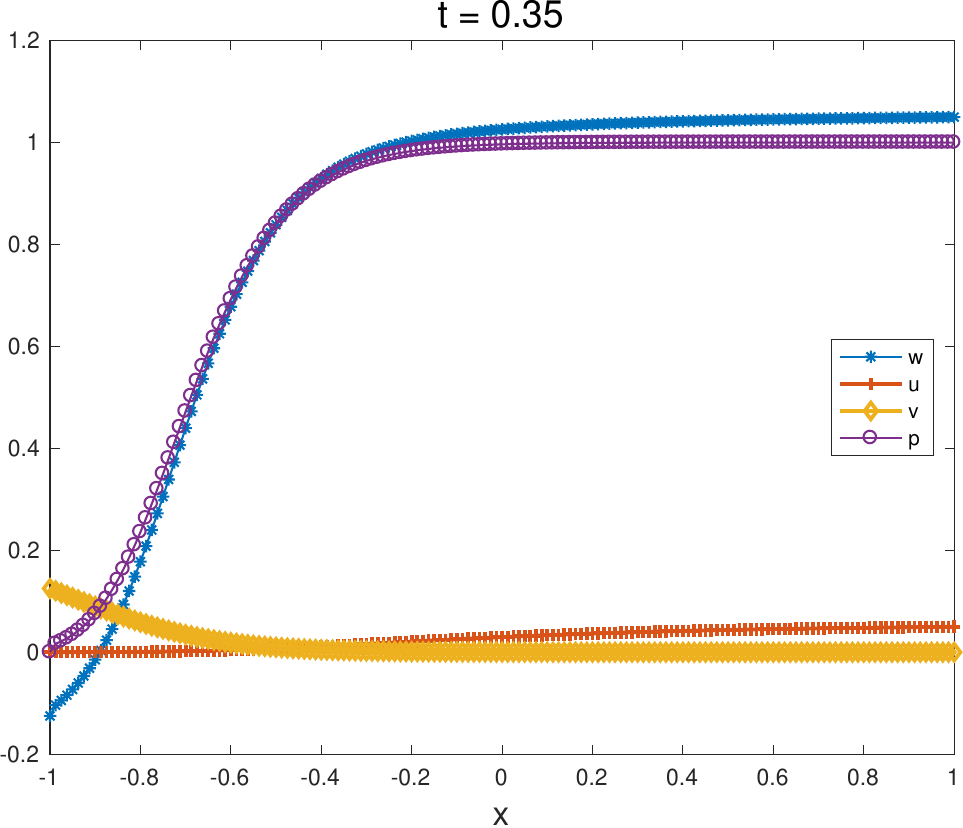}
	\includegraphics[width=5.0cm,height=5cm]{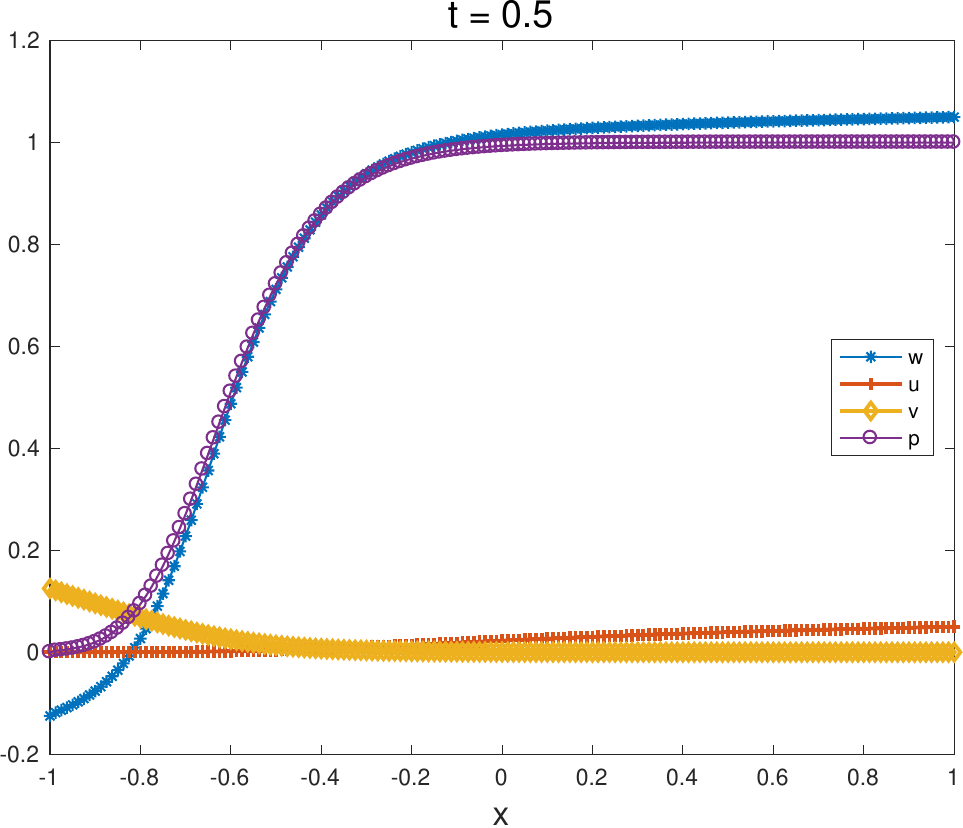}
	\includegraphics[width=5.0cm,height=5cm]{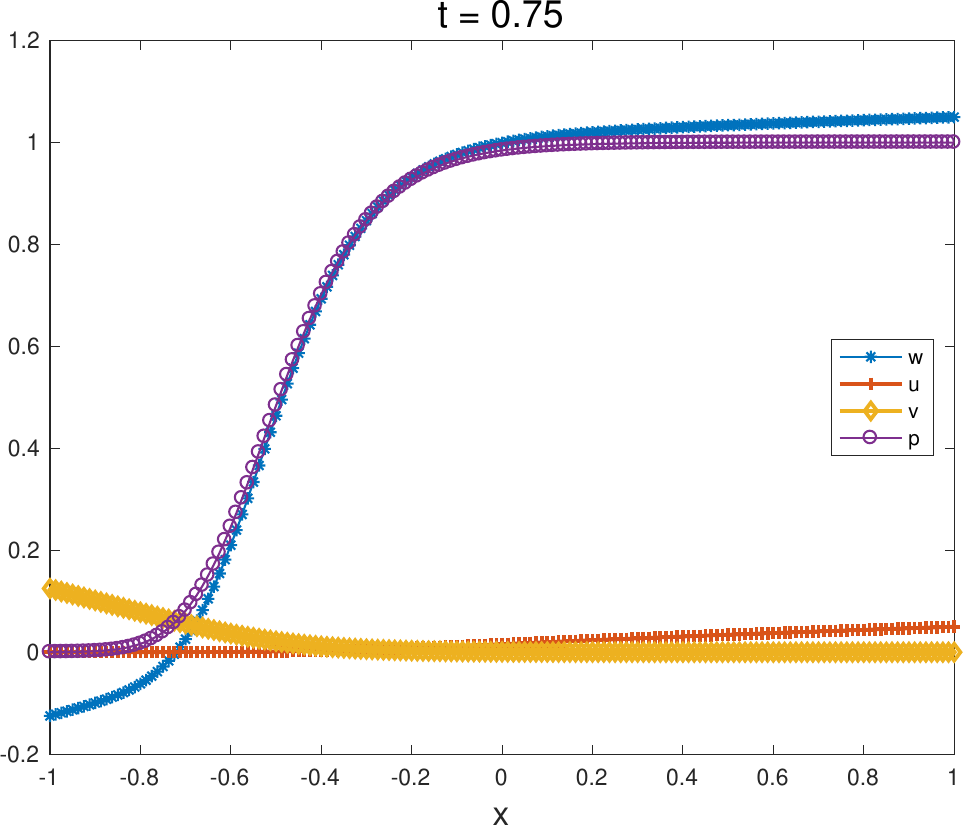}
	\includegraphics[width=5.0cm,height=5cm]{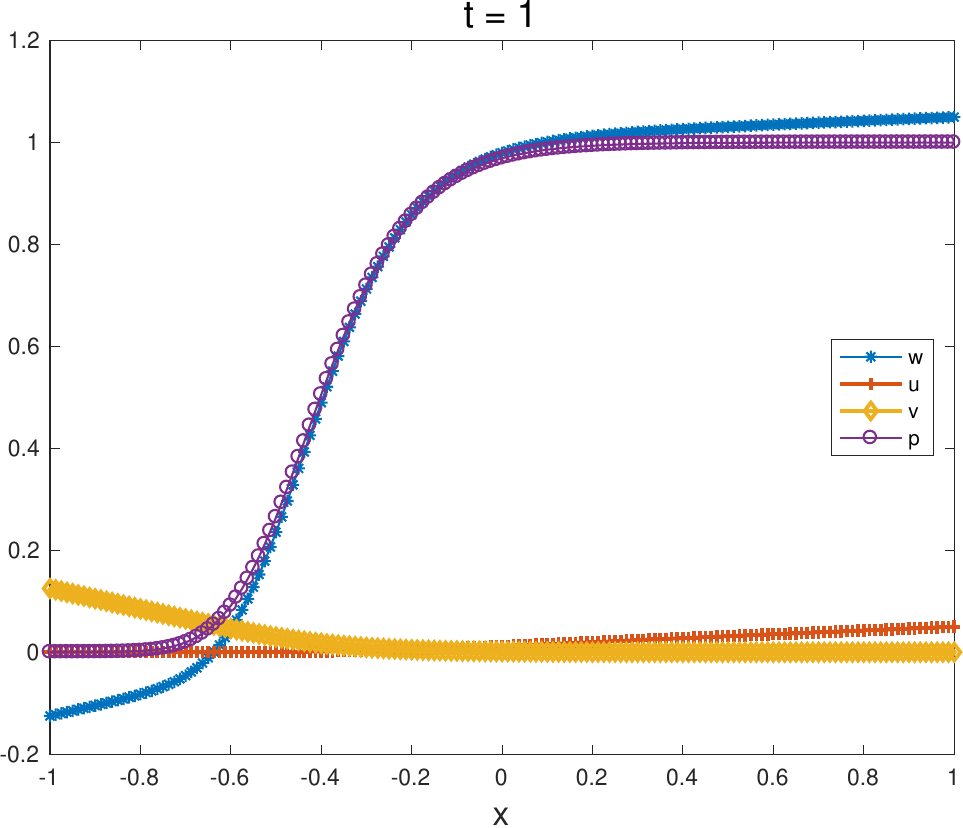}
	\caption{First-order semi-implicit scheme with $\epsilon = 10^{-2}$ and $h = 0.0125, \tau = 5 \times 10^{-6}$.}
	\label{plot: eppt01_limit}
\end{figure}
\begin{figure}[htp]
	\centering
	\includegraphics[width=5.0cm,height=5cm]{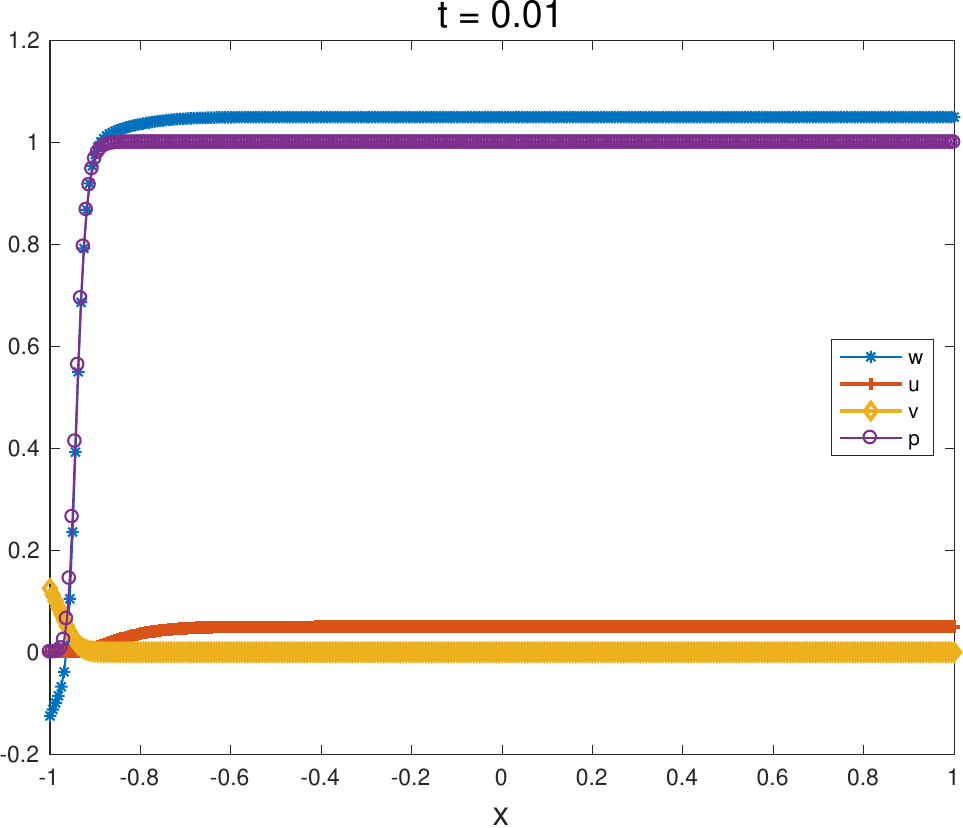}
	\includegraphics[width=5.0cm,height=5cm]{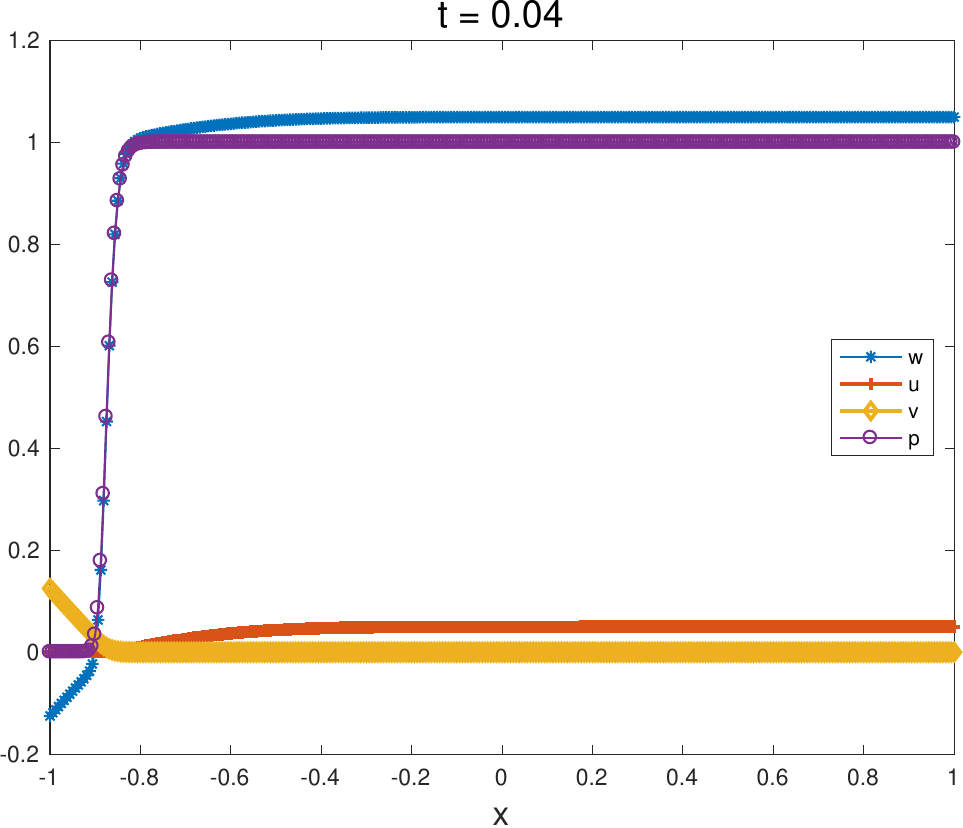}
	\includegraphics[width=5.0cm,height=5cm]{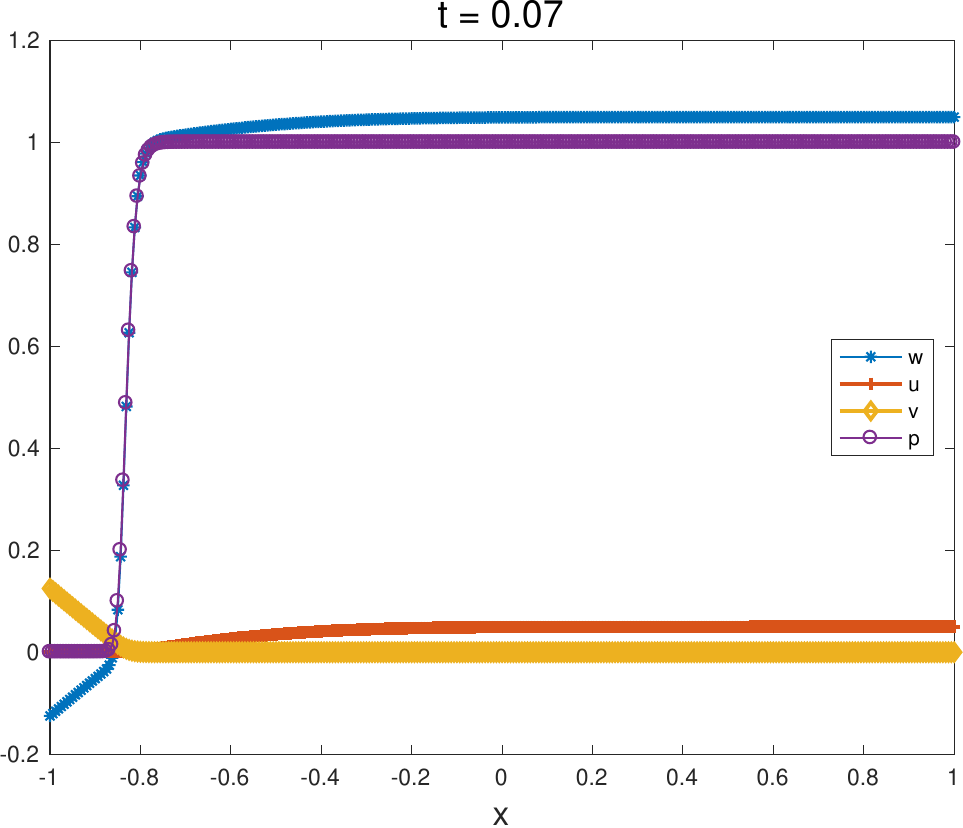}
	\includegraphics[width=5.0cm,height=5cm]{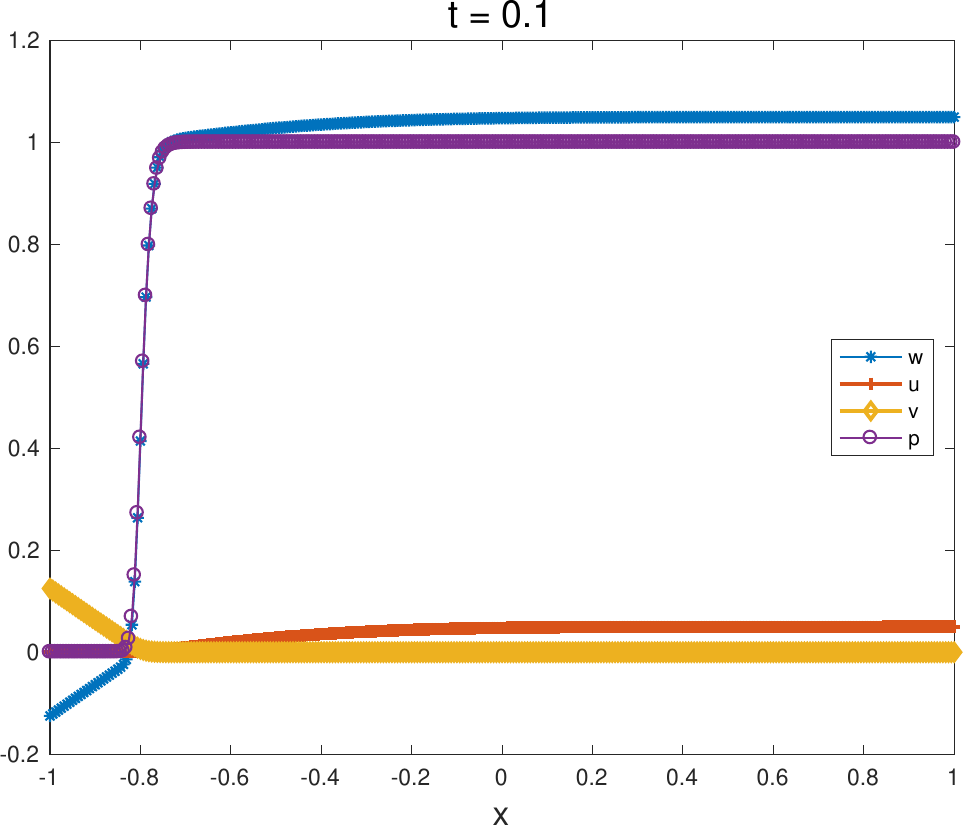}
	\includegraphics[width=5.0cm,height=5cm]{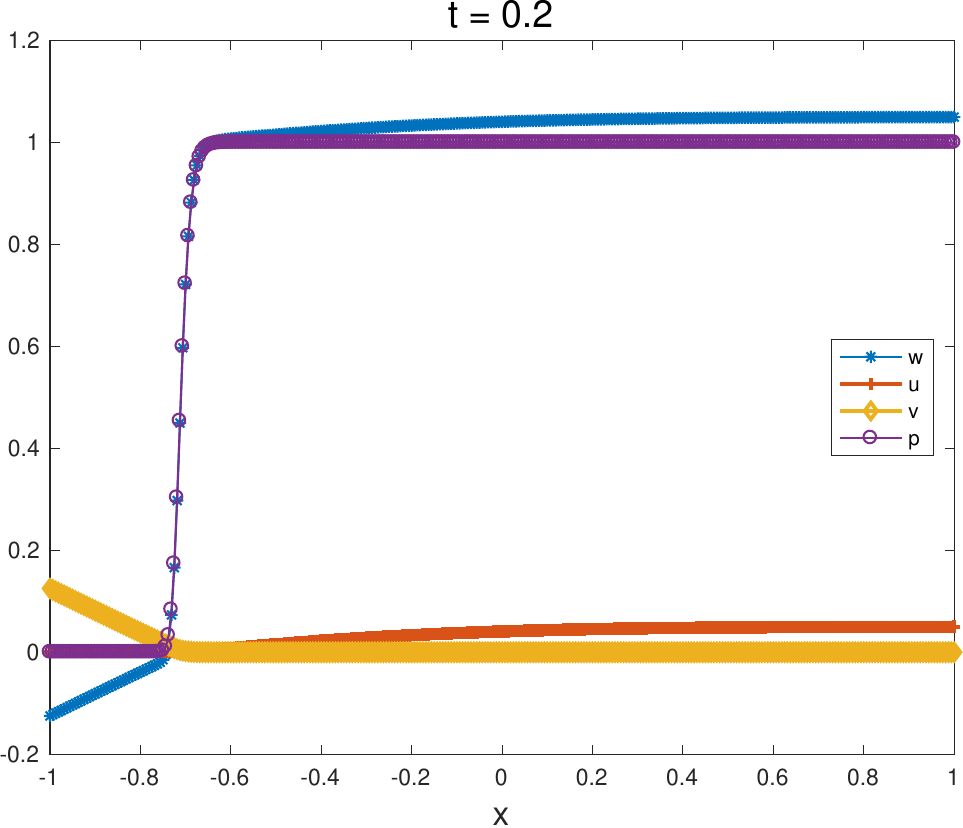}
	\includegraphics[width=5.0cm,height=5cm]{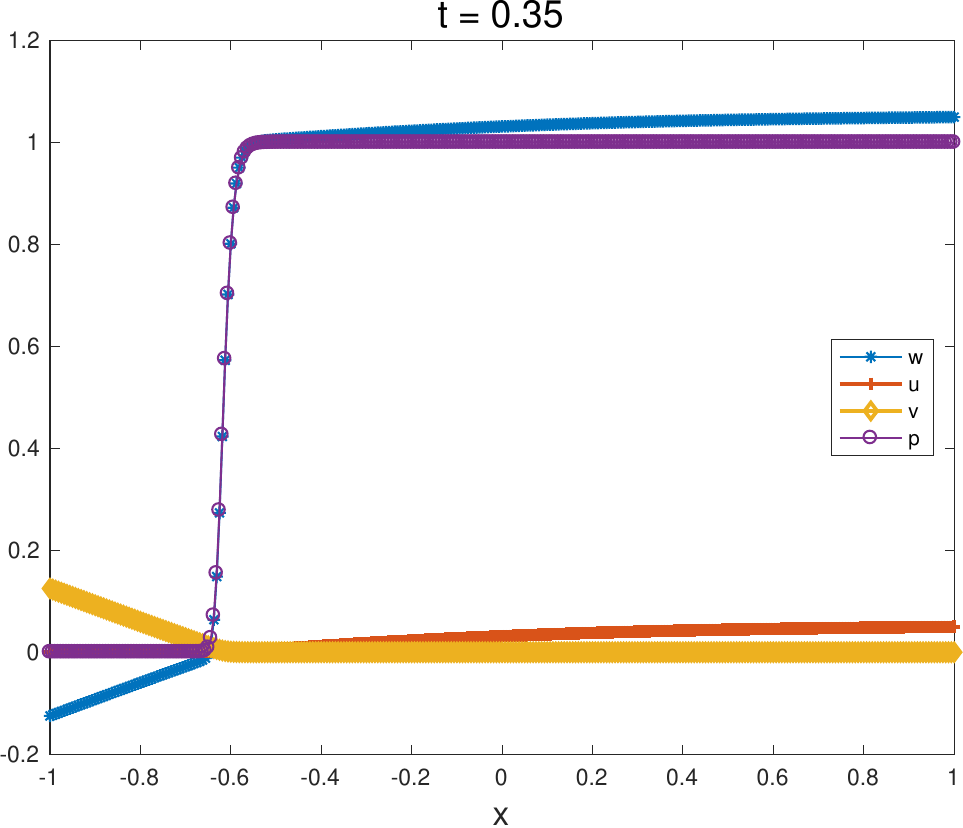}
	\includegraphics[width=5.0cm,height=5cm]{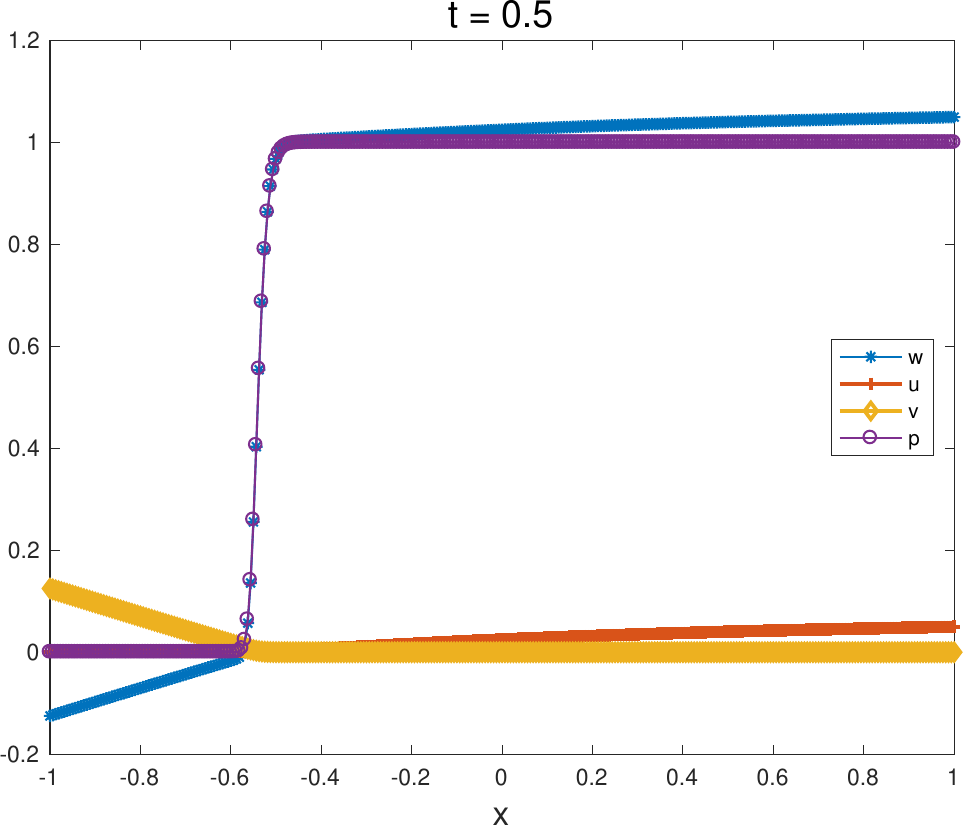}
	\includegraphics[width=5.0cm,height=5cm]{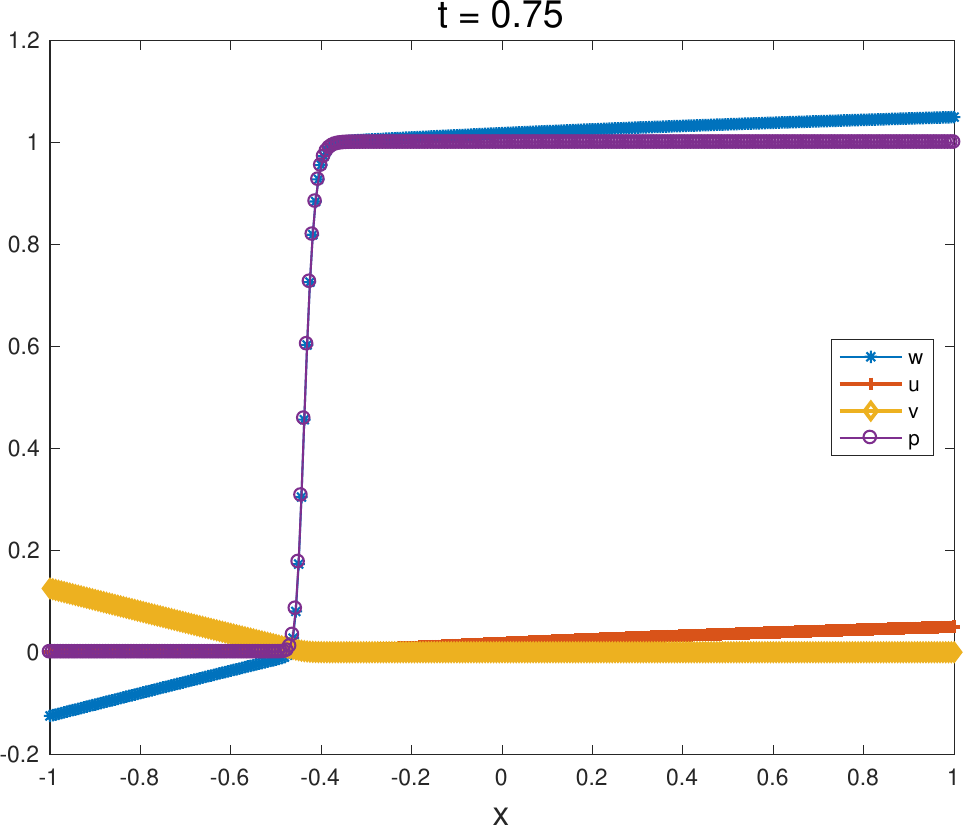}
	\includegraphics[width=5.0cm,height=5cm]{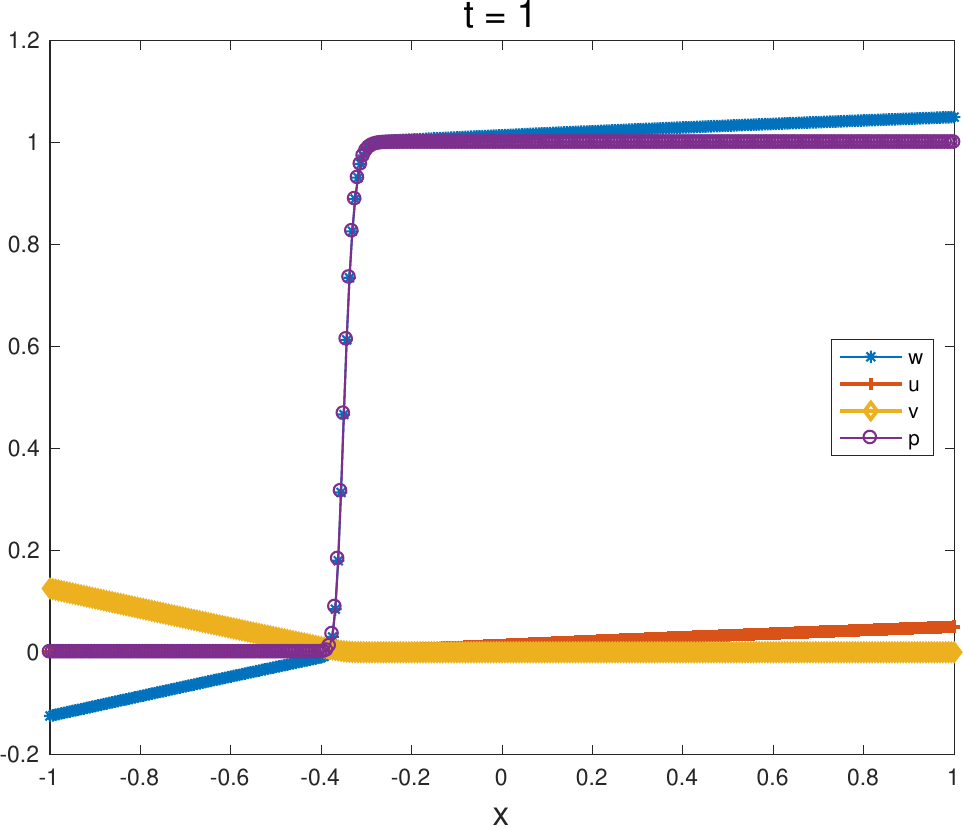}
	\caption{First-order semi-implicit scheme: solutions $u, v, p, w$ with $\epsilon = 10^{-4}$ and $h = 6.25 \times 10^{-3}$, $\tau = 5 \times 10^{-6}$.}
	\label{plot: eppt0001_limit}
\end{figure}

\begin{figure}[htp]
	\centering
	\includegraphics[width=10.0cm,height=9cm]{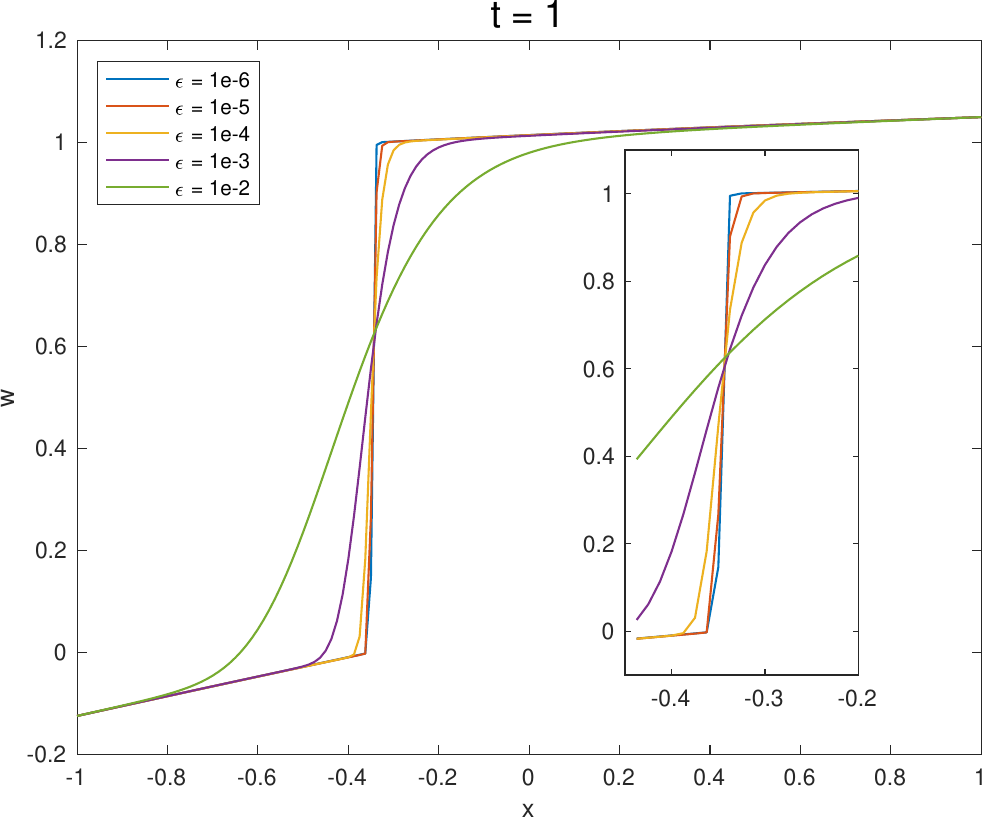}
	\caption{First-order semi-implicit scheme: the enthalpy function $w$ at $t = 1$ with various value of $\epsilon$ with $\tau = 5 \times 10^{-6}$ and $h = 0.0125$.}
	\label{plot: compare_eps}
\end{figure}
\begin{figure}[htp]
	\centering
	\includegraphics[width=8.0cm,height=7cm]{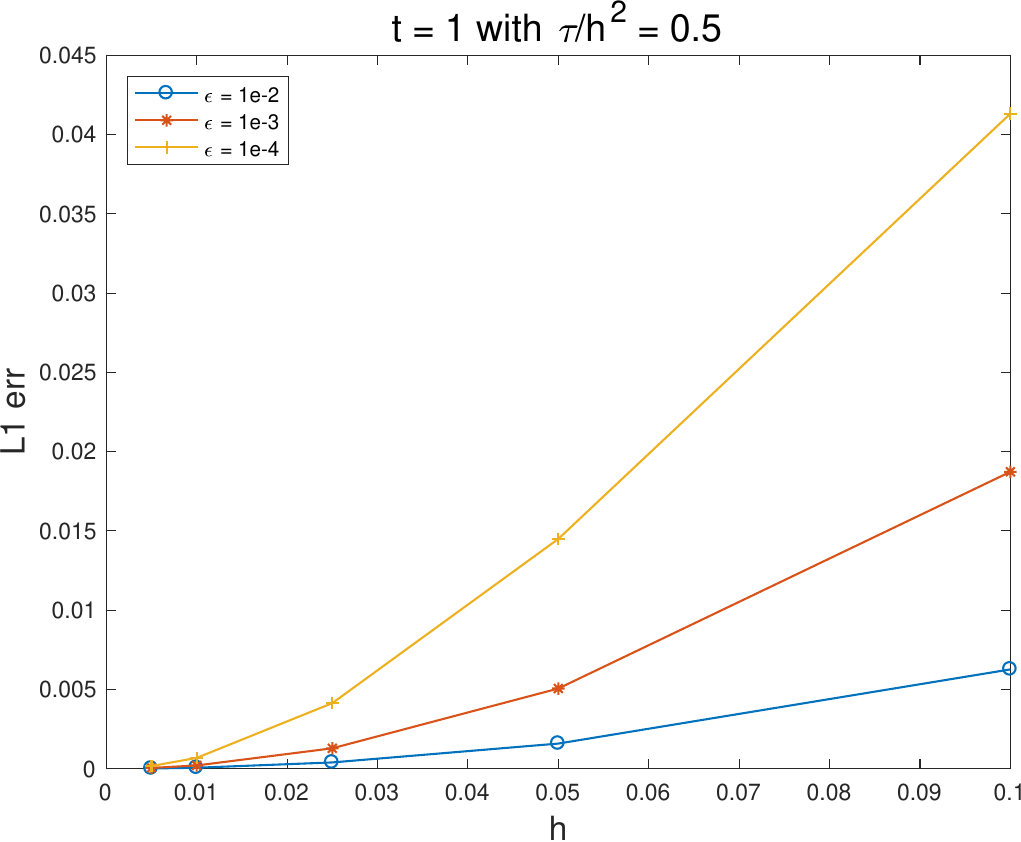}
	\includegraphics[width=8.0cm,height=7cm]{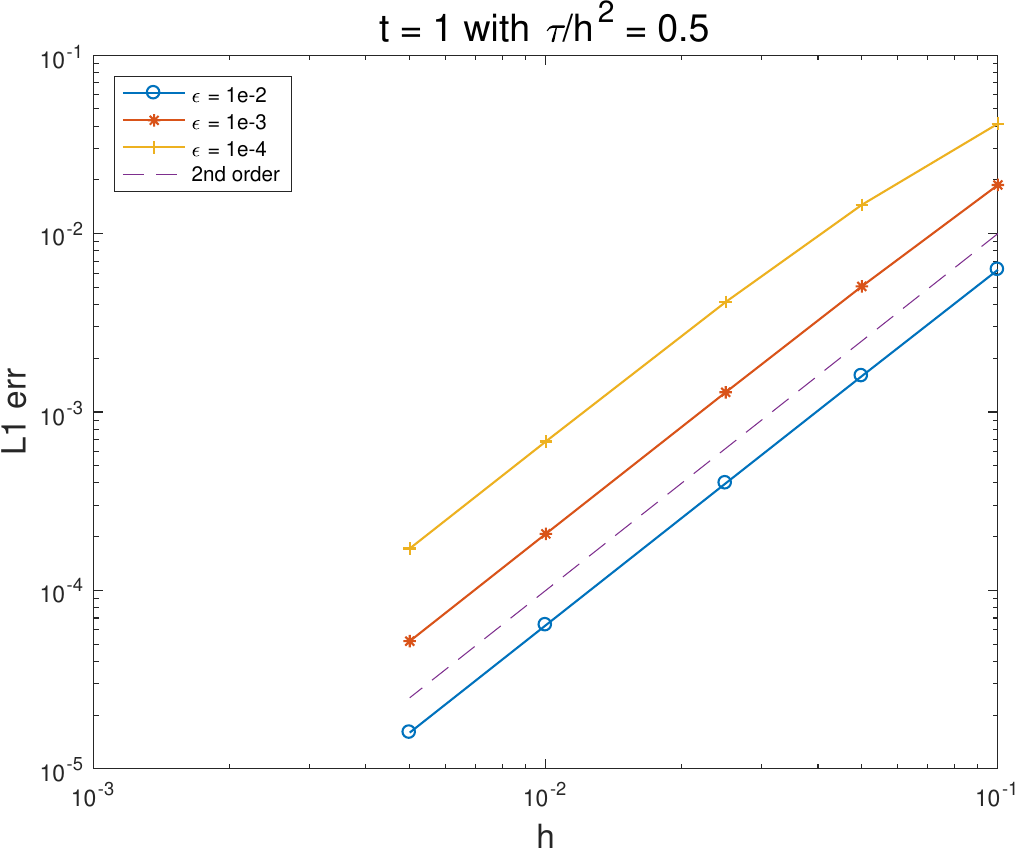}
	\caption{First-order semi-implicit scheme: $e_{w_{\tau,h}}^{1}$ at $t = 1$ with various value of $\epsilon$, $\tau = 0.01, 0.005, 0.0025, 0.001, 0.0005$ and $\tau/h^2 = 0.05$ where the reference solution is obtained by the fully-implicit method.}
	\label{plot: compare_eps_implicit}
\end{figure}

\subsubsection{Two-dimensional case}
In this test, we consider $\boldsymbol{x} = (x_1, x_2)^\prime \in [-L, L]\times[-L, L]$, $0 \le t \leqslant T$. The initial condition \eqref{i.c.} and Dirichlet boundary conditions \eqref{Dirichlet} given in the following form,
\begin{align}
	u_\epsilon(\boldsymbol{x}, 0) &= 
	\left\{
	\begin{array}{ll}
		0, & -L < x_1 \leq 0, \\
		\frac{\theta}{d_1}, & 0 < x_1 < L,
	\end{array}
	\right. \nonumber\\
	v_\epsilon(\boldsymbol{x}, 0) &= 
	\left\{
	\begin{array}{ll}
		\frac{\theta}{d_2}, & -L < x_1 \leq 0, \\
		0, & 0 < x_1 < L,
	\end{array}
	\right. \nonumber\\
	p_\epsilon(\boldsymbol{x}, 0) &= 
	\left\{
	\begin{array}{ll}
		0, & -L \leq x_1 \leq 0, \nonumber\\
		1, & 0 < x_1 \leq L,
	\end{array}
	\right. \nonumber\\
	u_\epsilon(\boldsymbol{x}, t)|_{x_1 = -L} &= 0, \quad u_\epsilon(\boldsymbol{x}, t)|_{x_1 = L} =  \frac{1}{2} (S_m - S_t) x_2 + \frac{1}{2} (S_m + S_t), \nonumber \\
	v_\epsilon(\boldsymbol{x}, t)|_{x_1 = -L} &= -\frac{1}{2} (S_m - S_t) x_2 + \frac{1}{2} (S_m + S_t), \quad v_\epsilon(\boldsymbol{x}, t)_{x_1 = L} = 0, \nonumber
\end{align}
where $d_1 = 1$, $d_2 = 2$, $\lambda = 1$, $S_t = 0.25$, $S_m = 0.05$, $L = 1$, $T = 3$ and 
\begin{align}
	\theta (\boldsymbol{x}, 0) = 
	\left\{
	\begin{array}{ll}
		-0.05, & -L < x_1 < 0, \\
		0, & x_1 = 0, \\
		0.05, & 0 < x_1 < L.
	\end{array}
	\right.
	\nonumber
\end{align}
Figure \ref{plot: semi_2d phi} and \ref{plot: semi_2d theta} give the results of the enthalpy function $w$ and the temperature $\theta = B(w)$ of the substance with $\epsilon = 10^{-5}$ and the mesh size $\Delta x = 0.05$, $\tau = 0.001$.

\begin{figure}[htp]
	\centering
	\includegraphics[width=4.0cm,height=4cm]{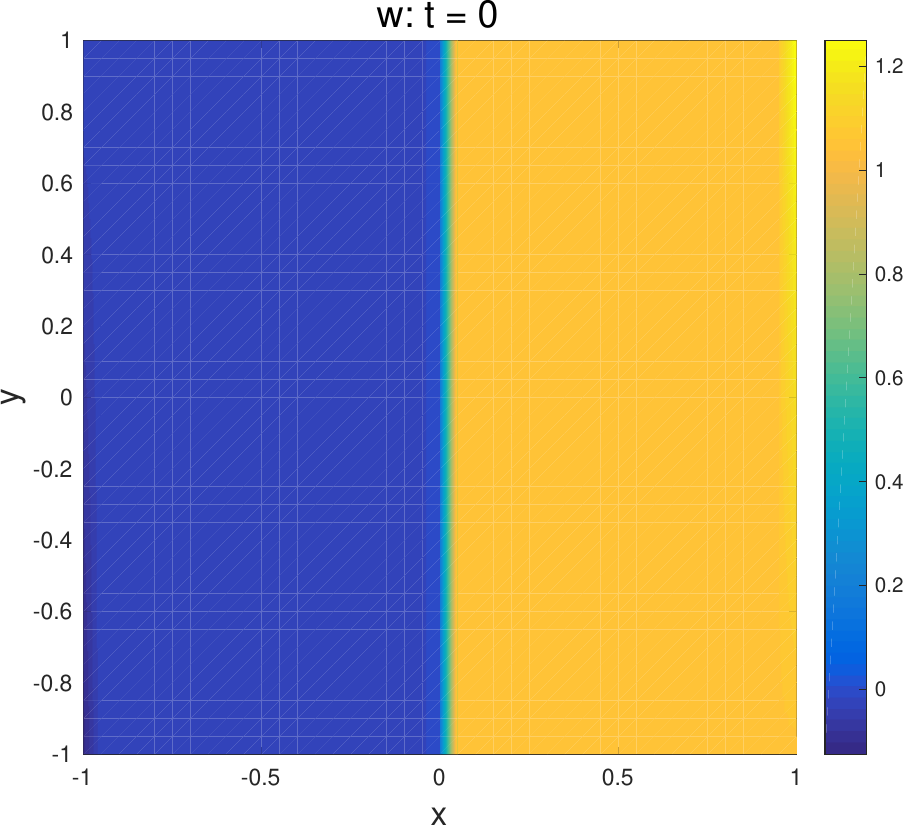}
	\includegraphics[width=4.0cm,height=4cm]{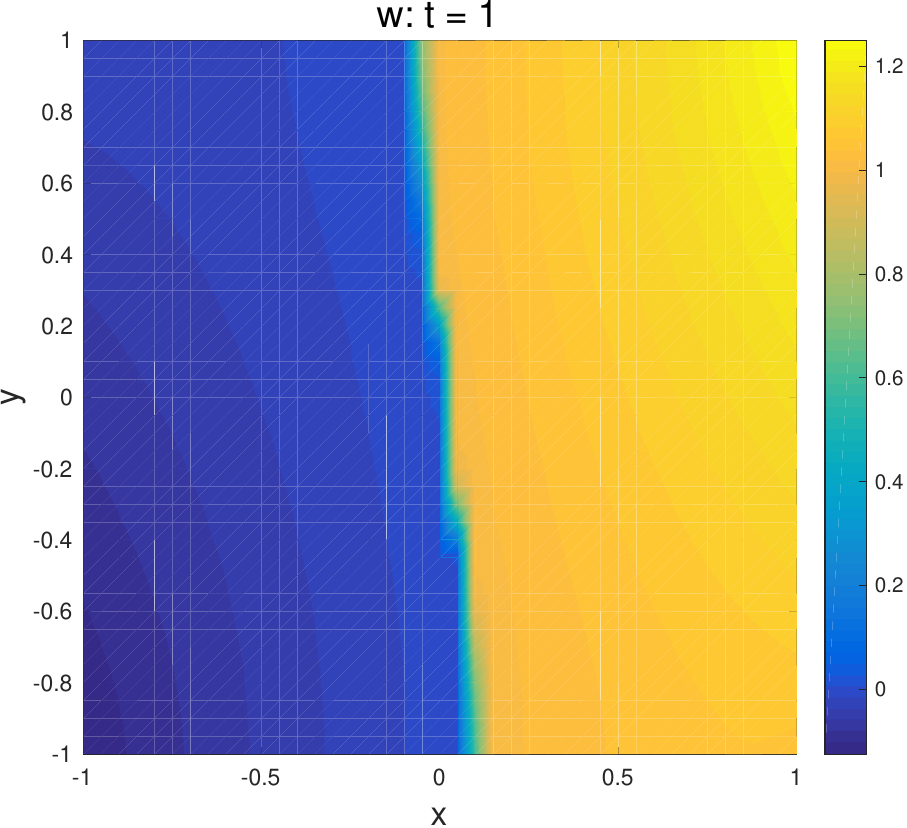}
	\includegraphics[width=4.0cm,height=4cm]{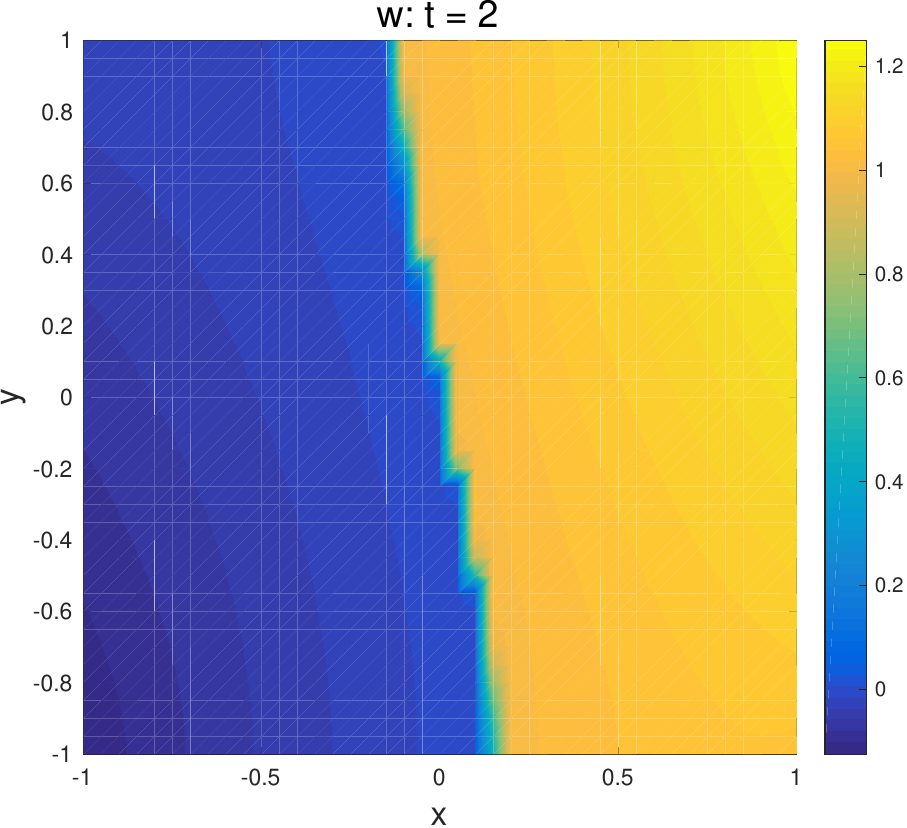}
	\includegraphics[width=4.0cm,height=4cm]{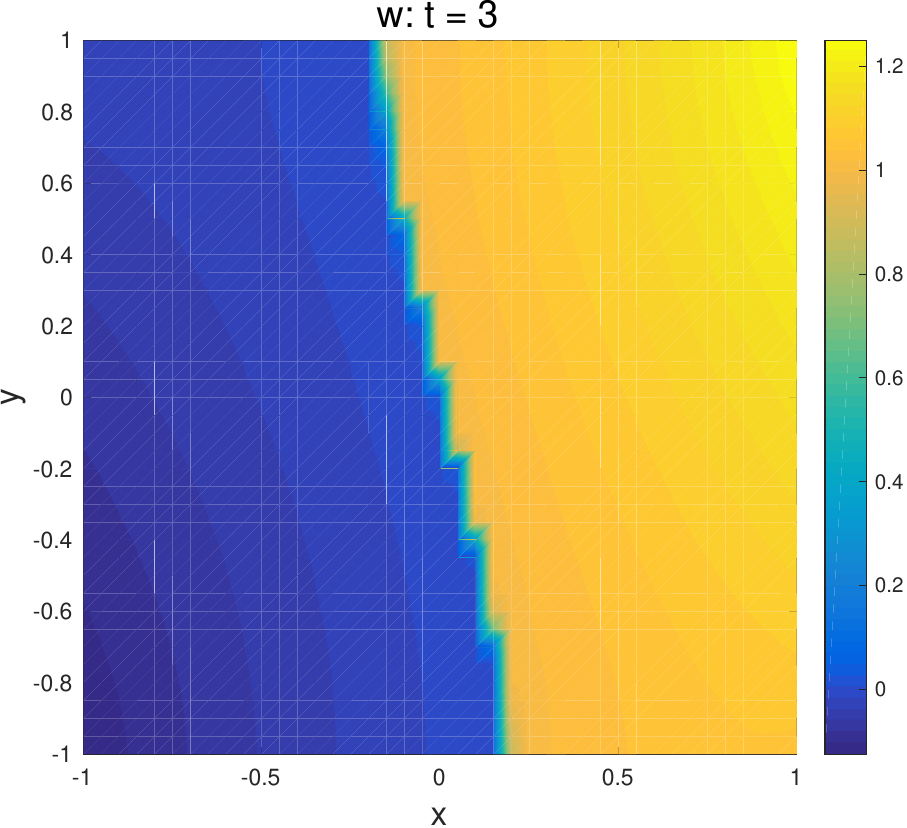}
	\caption{First-order semi-implicit scheme: $w$ with $\epsilon = 10^{-5}$ and $h = 0.05, \tau = 0.001$.}
	\label{plot: semi_2d phi}
\end{figure}

\begin{figure}[htp]
	\centering
	\includegraphics[width=4.0cm,height=4cm]{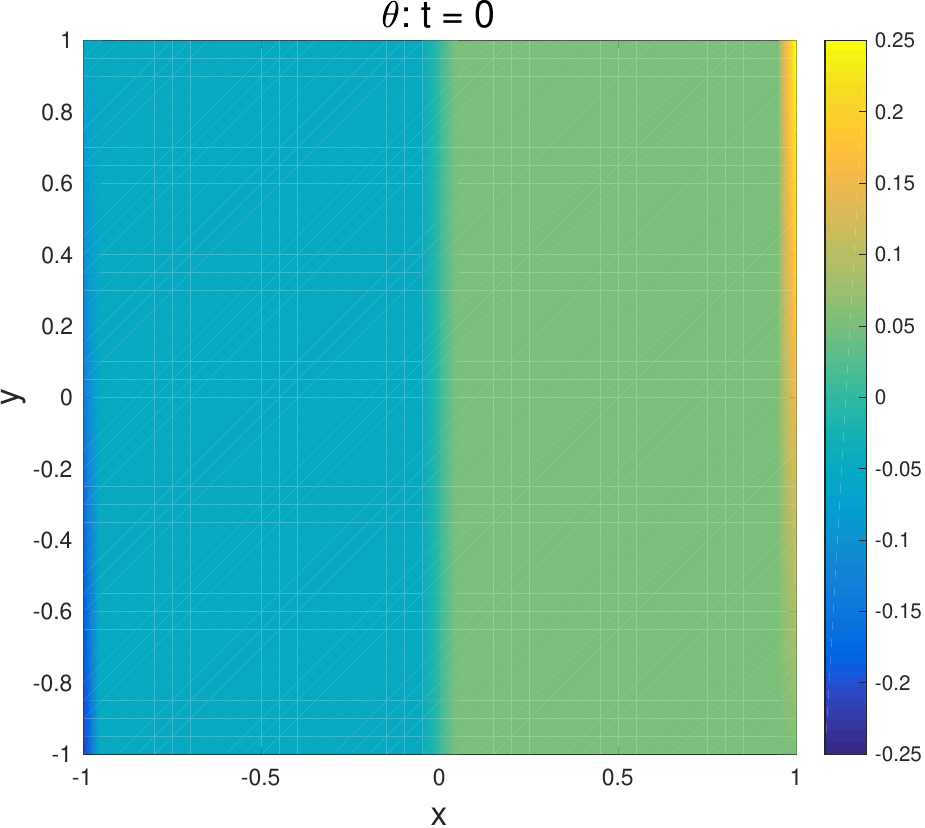}
	\includegraphics[width=4.0cm,height=4cm]{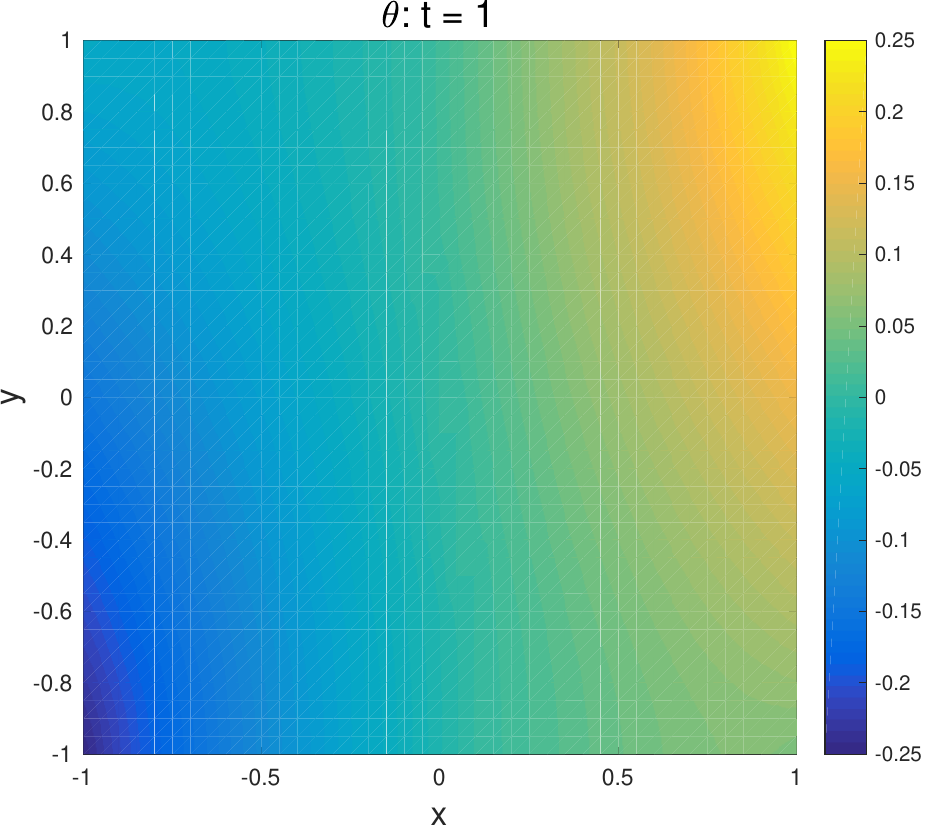}
	\includegraphics[width=4.0cm,height=4cm]{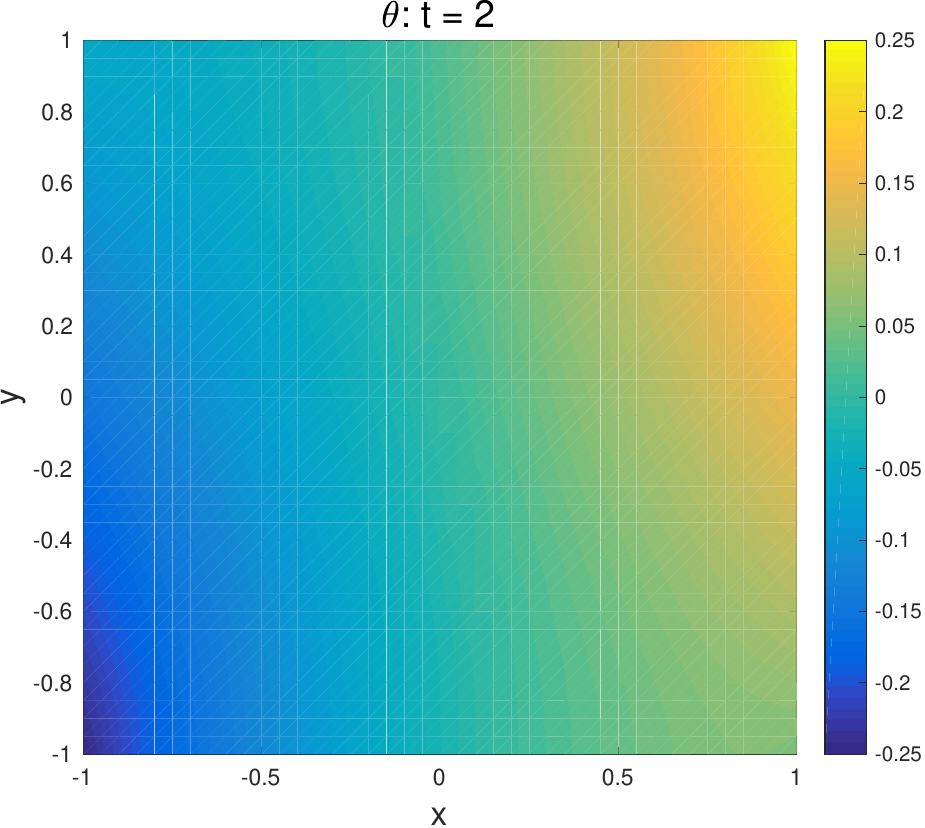}
	\includegraphics[width=4.0cm,height=4cm]{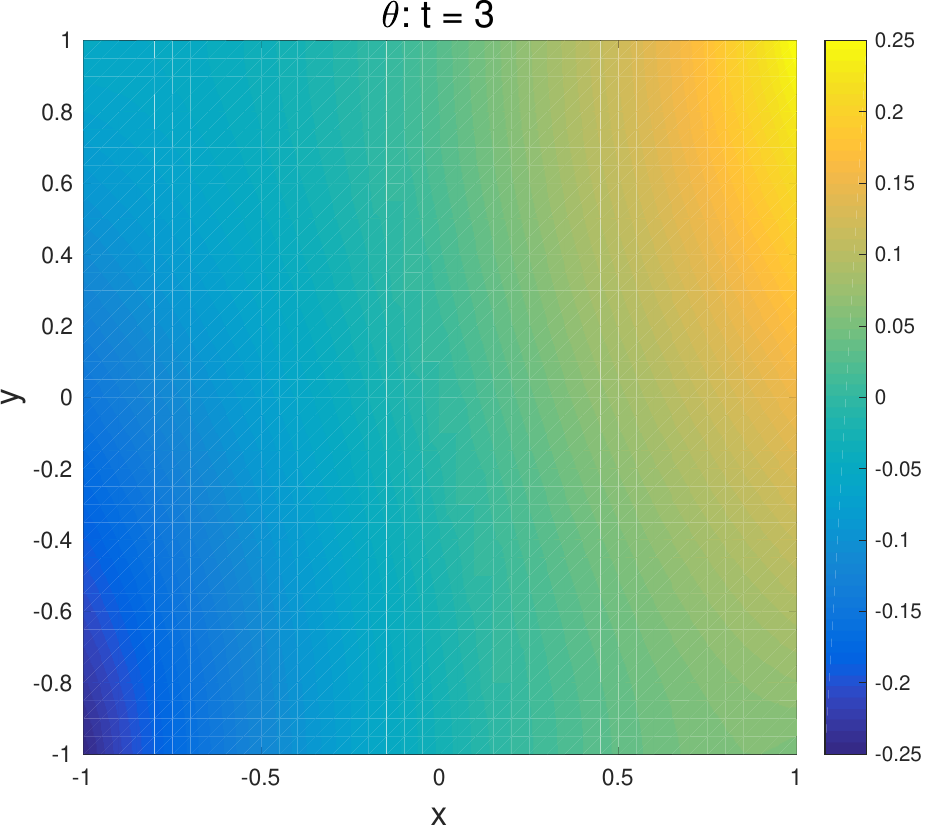}
	\caption{First-order semi-implicit scheme: $\theta = B(w)$ with $\epsilon = 10^{-5}$ and $h = 0.05, \tau = 0.001$.}
	\label{plot: semi_2d theta}
\end{figure}

\section{Conclusions}
In this paper, we study the semi-implicit strategy on the fast reaction--diffusion equations \eqref{eq:u}-\eqref{i.c.} and develop the semi-implicit schemes with first- and second-order accurate in time respectively. 
And our semi-implicit scheme preserves some analytical properties of the original model at the semi- and fully-discrete level. 
We also provide a series of numerical tests to demonstrate the properties, such as numerical convergence, non-negativity, and bound preserving, and to capture the movement of the sharp interface with various $\epsilon$ and to explore the heat transfer process, such as solid melting or liquid solidification.

\section*{Acknowledgements}
The work of Z. Zhou was supported by the National Key R \& D Program of China, Project Number 2020YFA0712000, and NSFC grant number 12031013, 12171013. 

\bibliographystyle{abbrv}
\bibliography{reference}

\end{document}